\documentclass[10pt]{amsart}
\usepackage{latexsym, amsmath, amsfonts, amssymb, mathrsfs, fancyhdr, tikz, tikz-cd}
\usepackage{verbatim}
\usepackage{multicol}
\usetikzlibrary{matrix,arrows,decorations.pathmorphing,calc}

\def\nin{\not \in}

\def\R{\mathbb{R}}

\def\s{\sigma}

\def\d{\delta}

\def\l{\lambda}

\def\V{\mathcal{V}}
\def\S{\mathcal{S}}

\def\a{\alpha}
\def\b{\beta}

\def\g{\gamma}

\def\la{\langle}
\def\ra{\rangle}

\def\ch2{\mathbb{C} \mathbb{H}^2}
\def\h2{\mathbb{H}^2}
\def\ddt{\frac{\partial}{\partial \theta}}
\def\ddr{\frac{\partial}{\partial r}}
\def\th{\theta}

\def\oshr/2{\cosh \left( \frac{r}{2} \right)}
\def\inhr/2{\sinh \left( \frac{r}{2} \right)}
\def\osh2r/2{\cosh^2 \left( \frac{r}{2} \right)}
\def\inh2r/2{\sinh^2 \left( \frac{r}{2} \right)}
\def\-1/4{- \frac{1}{4}}
\def\H{\mathbb{H}}
\def\C{\mathbb{C}}
\def\S{\mathbb{S}}
\def\P{\mathbb{P}}
\def\dr{\partial r}

\def\dds{\frac{\partial}{\partial \psi}}

\newtheorem{theorem}{Theorem}[section]
\newtheorem{lemma}[theorem]{Lemma}

\theoremstyle{definition}

\theoremstyle{remark}
\newtheorem{remark}[theorem]{Remark}

\numberwithin{equation}{section}

\begin{document}

\title{Warped product metrics on hyperbolic and complex hyperbolic manifolds}

\author{Barry Minemyer}
\address{Department of Mathematical and Digital Sciences, Commonwealth University - Bloomsburg campus, Bloomsburg, Pennsylvania 17815}
%\curraddr{Department of Mathematical Sciences,
%Binghamton University, Binghamton, New York 13902}
\email{bminemyer@commonwealthu.edu}
%\thanks{I was supported by the Mathematical Sciences Department at Binghamton University}

%\author{Author Two}
%\address{Mathematical Research Section, School of Mathematical Sciences,
%Australian National University, Canberra ACT 2601, Australia}

%    General info

\date{\today.}

%\dedicatory{This paper is dedicated to our authors.}


\begin{abstract}
In this paper we study warped-product metrics on manifolds of the form $X \setminus Y$, where $X$ denotes either $\H^n$ or $\C \H^n$, and $Y$ is a totally geodesic submanifold with arbitrary codimension.  
The main results that we prove are curvature formulas for these metrics on $X \setminus Y$ expressed in spherical coordinates about $Y$.  
We also discuss past and potential future applications of these formulas.
%We explain how these formulas relate to a recent result of Avramidi and Phan \cite{AP}, and discuss future applications of these formulas.
\end{abstract}

\maketitle

%\section*{This is an unnumbered first-level section head}
%This is an example of an unnumbered first-level heading.

%\specialsection*{Introduction}
%This is an example of a special section head.

\section{Introduction}\label{Section:Introduction}

\subsection{Main results}
Let $\H^n$ denote (real) $n$-dimensional hyperbolic space and let $\C \H^n$ denote (complex) $n$-dimensional complex hyperbolic space.  
In this paper, $X$ will denote either $\H^n$ or $\C \H^n$, and $Y$ will denote a totally geodesic submanifold of $X$.  
So if $X = \H^n$ then $Y = \H^k$, and if $X = \C \H^n$ then $Y$ is either $\H^k$ or $\C \H^k$ for some $0 \leq k \leq n-1$.  
Let $M$ be a Riemannian manifold, and $N$ a totally geodesic submanifold of $M$.  
We say that the pair $(M, N)$ is {\it modeled on} $(X, Y)$ if there exist lattices $\Gamma \subset \text{Isom}(X)$ and $\Lambda \subset \text{Isom}(Y)$ such that $M = X / \Gamma$, $N = Y / \Lambda$, and $\Lambda < \Gamma$.  
We also allow for the possibility that $N$ is disconnected.  
That is, we allow for multiple lattices $\Lambda < \Gamma$ which correspond to different (disjoint) copies of $\H^k \subset \H^n$ or $\H^k, \C \H^k \subset \C \H^n$.  
%The author, along with several collaborators, has been undergoing a systematic study of the geometry and topology of manifolds ``affiliated" with the pair $(M, N)$.  
%The purpose of this paper is to continue the development of the curvature formulas used in this research.  

The purpose of this paper is to develop curvature formulas for warped-product metrics on $X \setminus Y$ when the pair $(X, Y)$ is one of $(\H^n, \H^k)$, $(\C \H^n, \H^n)$, or $(\C \H^n, \C \H^k)$.  
These cases are detailed in Sections \ref{section:hn/hk}, \ref{section:chn/hn}, and \ref{section:chn/chk}, respectively.  
In each case we write the metric on $X$ in spherical coordinates about $Y$ (Theorems \ref{thm:hn/hk metric}, \ref{thm:chn/hn metric}, and \ref{thm:chn/chk metric}), we consider the corresponding warped product metric where we allow for variable coefficients in the metric tensor (equations \eqref{eqn:warped metric 1}, \eqref{eqn:warped metric 2}, and \eqref{eqn:/g}), and we compute formulas for the components of the $(4,0)$ curvature tensor with respect to these coefficient functions (Theorems \ref{thm:curvature tensor 0}, \ref{thm:curvature tensor 1}, and \ref{thm:curvature tensor 2}).  
These last three Theorems should be considered the main results of this paper.

\subsection{Applications for these curvature formulas}

Specific cases for these formulas are already known and have been used in various applications in the literature.  
Some examples are as follows.  
The case when $X = \H^n$ and $Y = \H^{n-2}$ was used by Gromov and Thurston in \cite{GT} (discussed further below) and by Belegradek in \cite{Belegradek real}.  
When $X = \H^n$ and $Y = \H^0$ is a point, this leads to the basis for the {\it Farrell and Jones Warping Deformation} used in \cite{FJ2}.  
This process is described by Ontaneda in \cite{Ontaneda 1} and used by the same author in \cite{Ontaneda 2}.  
The case when $X = \C \H^n$ and $Y = \C \H^0$ is a point was used by Farrell and Jones in \cite{FJ}, and the same $X$ but with $Y = \C \H^{n-1}$ was considered by Belegradek in \cite{Belegradek complex}.  
Finally, the cases when $(X,Y)$ are either $(\C \H^2, \H^2)$ or $(\C \H^n, \C \H^{n-2})$ were used by the author in \cite{Min} and \cite{Min negative distribution}.  

While the author believes that the curvature formulas in Theorems \ref{thm:curvature tensor 0}, \ref{thm:curvature tensor 1}, and \ref{thm:curvature tensor 2} will have many future uses, the primary motivation for the development of these curvature formulas was for the following application.

In \cite{GT} Gromov and Thurston famously construct pinched negatively curved manifolds which do not admit hyperbolic metrics.  
In this construction they consider pairs $(M, N)$ modeled on $(\H^n, \H^{n-2})$ which satisfy a few special topological and geometric conditions.
The pinched negatively curved manifold $B$ which does not admit a hyperbolic metric is then the $k$-fold cyclic branched cover of $M$ about $N$ (where $k \in \mathbb{N}$ can take all but possibly finitely many values).  
The difficulty in all of this is showing that $B$ exists, constructing a pinched negatively curved metric on $B$, and proving that $B$ does not admit a hyperbolic metric.

%Jean-Francois Lafont asked the author and several other collaborators whether or not this construction could be extended to the other locally symmetric pairs, namely to the pairs $(\C \H^n, \C \H^{n-2})$ and $(\C \H^2, \H^2)$.  
It is an open question as to whether or not this construction can be extended to the locally symmetric pairs $(\C \H^n, \C \H^{n-1})$ and $(\C \H^2, \H^2)$.
In forthcoming paper (\cite{Min complex metric}) the author shows that the $d$-fold cyclic ramified cover of $M$ about $N$ for the case $(\C \H^n, \C \H^{n-1})$ {\it does} admit an almost negatively $\frac{1}{4}$-pinched Riemannian metric.
The fact that such a pair $(M,N)$ can be realized so that the ramified cover is a smooth manifold for some integer $d > 2$ is a result of Stover and Toledo in \cite{ST}.  
The constructions of these Riemannian metrics are dependent on the curvature formulas proved in Theorems \ref{thm:curvature tensor 1} and \ref{thm:curvature tensor 2} below.

%A second, more minor, application of these curvature formulas is the following.  
%In \cite{AP} Avramidi and Phan prove that if $M$ is a complete finite volume Riemannian manifold with bounded nonpositive sectional curvature, then the ``thin part" of $M$ (the portion of $M$ ``heading off" toward the cusp) can only have nonzero homology up to dimension $\lfloor \frac{n}{2} \rfloor - 1$. 
%In particular, if $n=5$, then the thin part of $M$ must be aspherical.  
%In \cite{Min negative distribution} the author uses the formulas for the case $(\C \H^3, \C \H^1)$ to show that this result does not, in some sense, generalize to distributions.  
%Given any manifold of the form $M \setminus N$, where $(M, N)$ is modeled on $(\C \H^3, \C \H^1)$, there exists a Riemannian metric $g$ and a non-integrable $5$-dimensional distribution $\mathcal{D}$ where $g$ restricted to $\mathcal{D}$ satisfies all of the conditions above.  
%The ends of $M \setminus N$ are of the form $\S^3 \times \R^2$, and so if $\mathcal{D}$ were integrable then the corresponding submanifold with this inherited metric would violate the results in \cite{AP}.
%Of course, this metric can more generally be constructed on mainfolds $M \setminus N$ modeled on $(\C \H^n, \C \H^{n-2})$.

One last remark about these curvature formulas.  
In \cite{Belegradek real}, \cite{Belegradek complex}, and \cite{Min} it is proved that the manifold $M \setminus N$, where $(M, N)$ is modeled on one of $(\H^n, \H^{n-2})$, $(\C \H^n, \C \H^{n-2})$, or $(\C \H^2, \H^2)$, admits a complete, finite volume, negatively curved Riemannian metric.  
The curvature formulas developed in this paper generalize the curvature formulas computed and used in these three articles.

\subsection{Obstructions to $M \setminus N$ admitting a complete, finite volume Riemannian metric of negative sectional curvature} 
%Let $M \setminus N$ denote the $n$-dimensional manifold obtained from $M$ by ``drilling out" the totally geodesic submanifold $N$.
Consider the finite volume manifold $M \setminus N$.
The three cases where $N$ has real codimension two in $M$ are modeled on one of $(\H^n, \H^{n-2})$, $(\C \H^n, \C \H^{n-1})$, or $(\C \H^2, \H^2)$.  
In all of these cases, the manifold $M \setminus N$ admits a complete, finite volume Riemannian metric whose sectional curvature is bounded above by a negative constant (\cite{Belegradek real}, \cite{Belegradek complex}, and \cite{Min}).  

When the real codimension of $N$ is greater than two, the manifold $M \setminus N$ should not admit a complete, finite volume, negatively curved metric because it generally will not be aspherical.  
This fact should be realized in the curvature equations in Theorems \ref{thm:curvature tensor 0}, \ref{thm:curvature tensor 1}, and \ref{thm:curvature tensor 2}.  
More specifically, there should be an equation(s) which obstructs such a metric, but this (these) curvature equations should vanish when $N$ has codimension two.

In all cases except one ``exceptional case" the obstruction is a sectional curvature equation of the form
	\begin{equation}\label{eqn:failure}
	\frac{1}{v^2} - \left( \frac{v'}{v} \right)^2
	\end{equation}
where $v: \R \to \R$ is a positive, increasing real-valued function.  
In order to alter the metric on $M \setminus N$ to be complete, one needs to define a warping function for $v$ that will make each component of $N$ into the boundary of a cusp of the manifold.  
One easily checks that equation \eqref{eqn:failure} is nonpositive if and only if $1 \leq (v')^2$.  
But for the Riemannian metric to have any chance of having finite volume one needs $\lim_{r \to - \infty} v'(r) = 0$.

The one exceptional case is when $(M, N)$ is modeled on $(\C \H^n, \C \H^{n-2})$.  
Here, all curvature equations of the form \eqref{eqn:failure} vanish, and so this obstruction is more subtle.  
It should be noted that the vanishing of \eqref{eqn:failure} is what leads to the metric developed in \cite{Min negative distribution}.  
A detailed discussion about this situation is given in Subsection \ref{subsection:chn/chn-2}.

\subsection{Layout of this paper}
In Section \ref{section:hn/hk} we study manifolds of the form $\H^n \setminus \H^k$, in Section \ref{section:chn/hn} we consider $\C \H^n \setminus \H^n$, and in Section \ref{section:chn/chk} we analyze $\C \H^n \setminus \C \H^k$.    
The calculations in Section \ref{section:chn/hn} and \ref{section:chn/chk} become very complicated.
So in Section \ref{section:chn/hn} we restrict our attention to $\C \H^3 \setminus \H^3$ and in Section \ref{section:chn/chk} we restrict to $\C \H^5 \setminus \C \H^2$ to make these calculations simpler to follow.  
In each case, these are the smallest choices for $n$ and $k$ which capture all of the different formulas for the curvature tensor, up to the symmetries of the curvature tensor (and with respect to the frames chosen in each Section).
That is, from these cases one knows all of curvature formulas for general $\C \H^n \setminus \H^n$ and $\C \H^n \setminus \C \H^k$.
Also, notice that we only consider $\C \H^n \setminus \H^n$ in Section \ref{section:chn/hn} instead of the more general $\C \H^n \setminus \H^k$.  
The reason for this is due to simplicity:  in general there are several ways that $\H^k$ can sit inside of $\C \H^n$ which requires a case-by-case analysis.  
But in all situations this copy of $\H^k$ is contained in a copy of $\H^n$, and then one can apply our formulas here to $\C \H^n \setminus \H^n$.
Section \ref{section:preliminaries} is a short Section on some known formulas that are referenced throughout the paper, and Section \ref{section:proof of Lie brackets 1} is devoted to computing values for Lie brackets from Section \ref{section:chn/hn}.

We end this Section with the following two remarks which deal with notational differences between this paper and references \cite{Belegradek real}, \cite{Belegradek complex}, and \cite{Min}.

\begin{remark}\label{rmk:curvature bounds}
In this paper we scale the complex hyperbolic metric to have sectional curvatures in the interval $[-4,-1]$, whereas in the previous three references the curvatures were scaled to $\left[ -1,\-1/4 \right]$.  
To adjust the formulas in \cite{Belegradek real}, \cite{Belegradek complex}, and \cite{Min}, one simply multiplies the warping functions $h$, $v$, and $h_r$ by $\frac{1}{2}$.  
With this adjustment (and the following Remark), one sees that the formulas in these references agree with the codimension two versions of the formulas in Theorems \ref{thm:curvature tensor 0}, \ref{thm:curvature tensor 1}, and \ref{thm:curvature tensor 2}.
\end{remark}

\begin{remark}\label{rmk:curvature notation}
Another major notational difference between this paper and the papers \cite{Belegradek complex} and \cite{Belegradek real} is the formula used for the curvature tensor.
Let $g$ be a Riemannian metric with Levi-Civita connection $\nabla$, and let $W, X, Y,$ and $Z$ be vector fields.  
In this paper we follow \cite{do Carmo} and use the notation
	\begin{equation}\label{eqn:curvature tensor notation}
	R(X,Y)Z = \nabla_Y \nabla_X Z - \nabla_X \nabla_Y Z + \nabla_{[X,Y]} Z
	\end{equation}
for the curvature tensor $R$ of $g$.  
The negative of this formula is used in \cite{Belegradek complex} and \cite{Belegradek real}.
So, in particular, the $(4,0)$-curvature tensor $\la R(X,Y)Z,W \ra_g$ in this paper is equivalent to $\la R(X,Y)W,Z \ra_g$ in \cite{Belegradek complex} and \cite{Belegradek real}.  
\end{remark}

\vskip 20pt

%beginning of section 2
\section{Curvature formulas for warped product metrics on $\H^n \setminus \H^k$}\label{section:hn/hk}

\subsection{Expressing the metric in $\H^n$ in spherical coordinates about $\H^k$}
Let us first note that in Subsections 2.1, 3.1, and 4.1 we closely follow the notation and terminology used in \cite{Belegradek complex}.

Let ${\bf h_n}$ denote the hyperbolic metric on $\H^n$.  
Since $\H^k$ is a complete totally geodesic submanifold of the negatively curved manifold $\H^n$, there exists an orthogonal projection map $\pi : \H^n \to \H^k$.  
This map $\pi$ is a fiber bundle whose fibers are totally geodesic $(n-k)$-planes.

For $r > 0$ let $E(r)$ denote the $r$-neighborhood of $\H^k$.  
Then $E(r)$ is a real hypersurface in $\H^n$, and consequently we can decompose ${\bf h_n}$ as
	\begin{equation*}
	{\bf h_n} = ({\bf h_n})_r + dr^2
	\end{equation*}
where $({\bf h_n})_r$ is the induced Riemannian metric on $E(r)$.  
Let $\pi_r : E(r) \to \H^k$ denote the restriction of $\pi$ to $E(r)$.  
Note that $\pi_r$ is an $\S^{n-k-1}$-bundle whose fiber over any point $q \in \H^k$ is the $(n-k-1)$-sphere of radius $r$ in the totally geodesic $(n-k)$-plane $\pi^{-1}(q)$.  
The tangent bundle splits as an orthogonal sum $\mathcal{V}(r) \oplus \mathcal{H}(r)$ where $\mathcal{V}(r)$ is tangent to the sphere $\pi_r^{-1}(q)$ and $\mathcal{H}(r)$ is the orthogonal complement to $\mathcal{V}(r)$.

It is well known (see \cite{Belegradek real} or \cite{GT} when $k=n-2$ and \cite{Ontaneda} for general $k$) that for an appropriate identification of $E(r) \cong \H^k \times \S^{n-k-1}$ the metric $({\bf h_n})_r$ can be written as
	\begin{equation*}
	({\bf h_n})_r = \cosh^2(r) {\bf h_k} + \sinh^2(r) {\bf \sigma_{n-k-1}}
	\end{equation*}
where ${\bf h_k}$ denotes the hyperbolic metric on $\H^k$ and ${\bf \sigma_{n-k-1}}$ denotes the round metric on the unit sphere $\S^{n-k-1}$.  
Note that $({\bf h_n})_r$ restricted to $\mathcal{H}(r)$ is $\cosh^2(r) {\bf h_k}$ and $({\bf h_n})_r$ restricted to $\mathcal{V}(r)$ is $\sinh^2(r) {\bf \sigma_{n-k-1}}$.  
We summarize this in the following Theorem.

\vskip 20pt

\begin{theorem}\label{thm:hn/hk metric}
The hyperbolic manifold $\H^n \setminus \H^k$ can be written as $E \times (0, \infty)$  where $E \cong \H^k \times \mathbb{S}^{n-k-1}$ equipped with the metric
	\begin{equation}\label{eqn:hn/hk metric}
	{\bf h_n} = \cosh^2(r) {\bf h_k} + \sinh^2(r) \s^{n-k-1} + dr^2 .
	\end{equation}
\end{theorem}

\vskip 10pt

\subsection{The warped product metric and curvature formulas}
For some positive, increasing real-valued functions $h, v : (0,\infty) \to \R$ define
	\begin{equation}\label{eqn:warped metric 1}
	{\bf \l_r} := h^2(r) {\bf h_k} + v^2(r) {\bf \s_{n-k-1}}  \hskip 30pt \text{and} \hskip 30pt {\bf \l} := {\bf \l_r} + dr^2.
	\end{equation}
Of course, ${\bf \l} = {\bf h_n}$ when $h = \cosh(r)$ and $v = \sinh(r)$.  

Fix $p \in E(r)$ for some $r$ and let $q = \pi(p) \in \H^k$.  
Let $\{ \check{X}_i \}_{i=1}^k$ be an orthonormal frame of $\H^k$ near $q$ which satisfies $[\check{X}_i,\check{X}_j]_q = 0$ for all $1 \leq i, j \leq k$.
These vector fields can be extended to a collection of orthogonal vector fields $\{ X_i \}_{i=1}^k$ in a neighborhood of $p$ via the inclusion $\H^k \to E \times (0,\infty)$.  
Analogously, define an orthonormal frame $\{ \check{X}_j \}_{j=k+1}^{n-1}$ of $\S^{n-k-1}$ near (the projection of) $p$ which satisfies $[\check{X}_i,\check{X}_j]_p = 0$ for all $k+1 \leq i, j \leq n-1$, and extend this frame to vector fields $\{ X_j \}_{j=k+1}^{n-1}$ in a neighborhood of $p$ via the inclusion $\S^{n-k-1} \to E \times (0,\infty)$.  
Lastly, let $X_n = \ddr$.  

The orthogonal collection of vector fields $\{ X_i \}_{i=1}^n$ satisfies the following:
	\begin{enumerate}
	\item  $\la X_i,X_i \ra_\l = h^2$ for $1 \leq i \leq k$.
	\item  $\la X_i, X_i \ra_\l = v^2$ for $k+1 \leq i \leq n-1$.  
	\item  $\la X_n, X_n \ra_\l = 1$.
	\item  $[X_i,X_j]_p = 0$ for all $i, j$.
	\end{enumerate}
It should be noted that property (4) is special to the real hyperbolic case and will not be true in Sections \ref{section:chn/hn} and \ref{section:chn/chk} below.

Now define the corresponding orthonormal frame near $p$ by $Y_i = \frac{1}{h} X_i$ for $1 \leq i \leq k$, $Y_j = \frac{1}{v} X_j$ for $k+1 \leq j \leq n-1$, and $Y_n = X_n$.  
This frame satisfies the property that $[Y_i,Y_j]_p = 0$ for $1 \leq i, j \leq n-1$.
We can then apply formulas \eqref{eqn:Bel 1} through \eqref{eqn:Bel 4} to write the (4,0) curvature tensor $R_\l$ in terms of $R_{\l_r}$ as follows, where $1 \leq a, b \leq k$ and $k+1 \leq c, d \leq n-1$.  
	\begin{align*}
	&K_\l (Y_a,Y_b) = K_{\l_r}(Y_a,Y_b) - \left( \frac{h'}{h} \right)^2    \hskip 30pt  K_\l (Y_c,Y_d) = K_{\l_r}(Y_c,Y_d) - \left( \frac{v'}{v} \right)^2   \\
	&K_\l (Y_a,Y_c) = K_{\l_r}(Y_a,Y_c) - \frac{h'v'}{hv}    \hskip 20pt  K_\l (Y_a,Y_n) = - \frac{h''}{h}  \hskip 20pt   K_\l(Y_c,Y_n) = - \frac{v''}{v}.
	\end{align*}
In the above equations, we use the notation 
	$$ K(X, Y) = \la R(X, Y)X, Y \ra $$
to denote the sectional curvature of the 2-plane spanned by $X$ and $Y$.
The above equations are the only terms that appear (up to the symmetries of the curvature tensor).  
So, in particular, all mixed terms of $R_\l$ are identically zero.

Now, the $(4,0)$ curvature tensor $R_{\l_r}$ is simple to calculate.  
Since both $h(r) \H^k$ and $v(r) \S^{n-k-1}$ have constant curvature, and $h(r) \H^k \times v(r) \S^{n-k-1}$ is metrically a product, we have that for $1 \leq a, b, \leq k$ and $k+1 \leq c, d \leq n-1$:
	\begin{align*}
	K_{\l_r} (Y_a,Y_b) = - \frac{1}{h^2} \hskip 30pt K_{\l_r}(Y_c,Y_d) = \frac{1}{v^2} \hskip 30pt K_{\l_r}(Y_a,Y_c) = 0.
	\end{align*}
Putting this all together yields the following.

\vskip 20pt

\begin{theorem}\label{thm:curvature tensor 0}
Up to the symmetries of the curvature tensor, the only nonzero terms of the $(4,0)$ curvature tensor $R_\l$ are:
	\begin{align*}
	&K_\l (Y_a,Y_b) = - \frac{1}{h^2} - \left( \frac{h'}{h} \right)^2    \hskip 40pt  K_\l (Y_c,Y_d) = \frac{1}{v^2} - \left( \frac{v'}{v} \right)^2   \\
	&K_\l (Y_a,Y_c) = - \frac{h'v'}{hv}    \hskip 20pt  K_\l (Y_a,Y_n) = - \frac{h''}{h}  \hskip 20pt   K_\l(Y_c,Y_n) = - \frac{v''}{v}
	\end{align*}
where $1 \leq a, b, \leq k$ and $k+1 \leq c, d \leq n-1$.
\end{theorem}

\vskip 10pt

One easily checks that plugging in the values $v(r) = \sinh(r)$ and $h(r) = \cosh(r)$ gives all sectional curvatures of $-1$.

\vskip 20pt

%beginning of section 3
\section{Curvature formulas for warped product metrics on $\C \H^n \setminus \H^n$}\label{section:chn/hn}
As mentioned in the Introduction, for simplicity we are going to restrict ourselves to the case when $n=3$.  
This is exactly the smallest dimension which captures every nonzero component of the curvature tensor, and so nothing is lost with this restriction (see the comments after Theorem \ref{thm:curvature tensor 1} for more discussion).

\subsection{Expressing the metric in $\C \H^3$ in spherical coordinates about $\H^3$}
Let ${\bf c_3}$ denote the complex hyperbolic metric on $\C \H^3$ normalized to have constant holomorphic sectional curvature $-4$.  
Since $\H^3$ is a complete totally geodesic submanifold of the negatively curved manifold $\C \H^3$, there exists an orthogonal projection map $\pi : \C \H^3 \to \H^3$.  
This map $\pi$ is a fiber bundle whose fibers are totally real totally geodesic $3$-planes, and therefore have constant sectional curvature $-1$.  

For $r > 0$ let $E(r)$ denote the $r$-neighborhood of $\H^3$.  
Then $E(r)$ is a real hypersurface in $\C \H^3$, and consequently we can decompose ${\bf c_3}$ as
	\begin{equation*}
	{\bf c_3} = ({\bf c_3})_r + dr^2
	\end{equation*}
where $({\bf c_3})_r$ is the induced Riemannian metric on $E(r)$.  
Let $\pi_r : E(r) \to \H^3$ denote the restriction of $\pi$ to $E(r)$.  
Note that $\pi_r$ is an $\S^2$-bundle whose fiber over any point $q \in \H^3$ is the 2-sphere of radius $r$ in the totally real totally geodesic $3$-plane $\pi^{-1}(q)$.  
The tangent bundle splits as an orthogonal sum $\mathcal{V}(r) \oplus \mathcal{H}(r)$ where $\mathcal{V}(r)$ is tangent to the 2-sphere $\pi_r^{-1}(q)$ and $\mathcal{H}(r)$ is the orthogonal complement to $\mathcal{V}(r)$.

For $r, s > 0$ there exists a diffeomorphism $\phi_{sr}: E(s) \to E(r)$ induced by the geodesic flow along the totally real totally geodesic $3$-planes orthogonal to $\H^3$.   
Fix $p \in E(r)$ arbitrary, let $q = \pi(p) \in \H^3$, and let $\g$ be the unit speed geodesic such that $\g(0) = q$ and $\g(r) = p$. 
In what follows, all computations are considered in the tangent space $T_p E(r)$. 

Note that $\V(r)$ is tangent to both $E(r)$ and the totally real totally geodesic $3$-plane $\pi^{-1}(q)$.  
Then since $\pi^{-1}(q)$ is preserved by the geodesic flow, we have that $d \phi_{sr}$ takes $\V(s)$ to $\V(r)$.  
Since $exp_p^{-1} \left( \pi^{-1}(q) \right)$ is a totally real 3-plane, there exists a suitable identification $\pi^{-1}(q) \cong \S^2 \times (0,\infty)$ where the metric ${\bf c_3}$ restricted to $\pi^{-1}(q)$ can be written as 
	\begin{equation*}
	\sinh^2(r) \s^2 + dr^2.
	\end{equation*}
Here, $\s^2$ is the round metric on the unit 2-sphere.

Let 
	\begin{equation}\label{eqn:X4 and X5 check}
	\check{X}_4 = \ddt	\hskip 40pt	\check{X}_5 = \frac{1}{\sin \th} \dds
	\end{equation}
be an orthonormal frame on a neighborhood of (the projection of) $p$ in $\S^2$, and extend these to orthogonal vector fields $\{ X_4, X_5 \}$ on $\pi^{-1}(q)$ via the inclusion $\S^2 \to \pi^{-1}(q)$.  
 Note that both $X_4$ and $X_5$ are invariant under $d \phi_{sr}$.  
 Let $X_6 = \ddr$.

Let $J$ denote the complex structure on $\C \H^3$.  
It is well known that $J_p$ preserves complex subspaces in $T_p \C \H^3$ and maps real subspaces into their orthogonal complement.
Since $\left( X_4, X_5, X_6 \right)$ spans a real $3$-plane in $T_p \C \H^3$, its orthogonal complement $\mathcal{H}_p(r)$ is spanned by $\left( JX_4, JX_5, JX_6 \right)$.  
In what follows we define vector fields $X_1$, $X_2$, and $X_3$ which are just scaled copies of $JX_4$, $JX_5$, and $JX_6$, respectively.

\subsubsection{The vector fields $X_1$ and $X_2$}
First note that $\left( JX_4, X_6 \right)$ spans a real $2$-plane in $T_p \C \H^3$ (since its $J$-image is contained in its orthogonal complement).  
So $P = \text{exp} \left( \text{span} \left( JX_4, X_6 \right) \right)$ is a totally real totally geodesic $2$-plane in $\C \H^3$ which intersects $\H^3$ orthogonally.
Since this intersection is orthogonal, $P$ is preserved by the geodesic flow $\phi$.  
Therefore, $\text{span} \left( JX_4 \right)$ is preserved by $d \phi$.  

The set $P \cap \H^3$ is a (real) geodesic.  
Let $\a(s)$ denote this geodesic parameterized with respect to arc length so that $\a(0) = q$.  
Then define $(X_1)_p = (d \pi)^{-1}_p \a'(0)$.  
There exists a positive real-valued function $a(r,s)$ so that the metric ${\bf c_3}$ restricted to $P$ is of the form $dr^2 + a^2(r,s) ds^2$.  
But since $\R$ acts by isometries on $P$ via translation along $\a$, the function $a(r,s)$ is independent of $s$.
%But note that $a(r,s)$ is independent of $s$ because of the isometric $\R$-action on $P$ by translations along $\a$.  
Then since the curvature of a real $2$-plane is $-1$, we have that $a(r) = \cosh(r)$.  

We analogously define $X_2$ by replacing $X_4$ with $X_5$ in the above description.  
All conclusions follow in an identical manner.  
Thus, we can write the metric ${\bf c_3}$ restricted to $exp_p(X_1,X_2,X_6)$ as
	\begin{equation*}
	\cosh^2(r) (dX_1^2 + dX_2^2 ) + dr^2.
	\end{equation*}

\subsubsection{The vector field $X_3$}
This is also mostly analogous to the definition of $X_1$.
But this time note that $\left( JX_6,X_6 \right)$ spans a complex line in $T_p \C \H^3$ (since it is preserved by its $J$-image).
So $Q = \text{exp}_p \left( \text{span} \left( JX_6,X_6 \right) \right)$ is a complex geodesic in $\C \H^3$ which intersects $\H^3$ orthogonally.
Since this intersection is orthogonal, $Q$ is preserved by the geodesic flow $\phi$.  
Therefore, $\text{span} \left( JX_6 \right)$ is preserved by $d \phi$.  

The set $Q \cap \H^3$ is a (real) geodesic.  
Let $\b(t)$ denote this geodesic parameterized with respect to arc length so that $\b(0) = q$.
Then define $(X_3)_p = (d \pi)^{-1} \b'(0)$.  
There exists a positive real-valued function $b(r,t)$ so that the metric ${\bf c_3}$ restricted to $Q$ is of the form $dr^2 + b^2(r,t) dt^2$.
But since $\R$ acts by isometries on $Q$ via translation along $\b$, the function $b(r,t)$ is independent of $t$.
%But note that $b(r,t)$ is independent of $t$ because of the isometric $\R$-action on $Q$ by translations along $\b$. 
Then since the curvature of a complex geodesic is $-4$, we have that $b(r) = \cosh(2r)$.

\subsubsection{Conclusion}

\vskip 20pt

\begin{theorem}\label{thm:chn/hn metric}
The complex hyperbolic manifold $\C \H^3 \setminus \H^3$ can be written as $E \times (0, \infty)$  where $E \cong \H^3 \times \mathbb{S}^2$ equipped with the metric
	\begin{equation}\label{eqn:chn/hn metric}
	{\bf c_3} = \cosh^2(r) (dX_1^2 + dX_2^2) + \cosh^2(2r) dX_3^2 + \sinh^2(r) (dX_4^2 + dX_5^2) + dr^2 .
	\end{equation}
\end{theorem}

\vskip 10pt

In equation \eqref{eqn:chn/hn metric}, $dX_1$ through $dX_5$ denote the covector fields dual to the vector fields $X_1$ through $X_5$, respectively.
Lastly, notice that $dX_1^2 + dX_2^2$ is the hyperbolic metric with constant sectional curvature $-1$, and $dX_4^2 + dX_5^2$ is the spherical metric with constant sectional curvature $1$.

\subsection{The warped product metric and curvature formulas in $\C \H^3 \setminus \H^3$}
For some positive, increasing real-valued functions $h, h_r, v : (0,\infty) \to \R$ define
	\begin{equation*}
	{\bf \mu_r} := h^2(r) (dX_1^2 + dX_2^2) + h_r^2(r) dX_3^2 + v^2(r) (dX_4^2 + dX_5^2)    
	\end{equation*}
and
	\begin{equation}\label{eqn:warped metric 2}
	{\bf \mu} := {\bf \mu_r} + dr^2.
	\end{equation}
Of course, ${\bf \mu} = {\bf c_3}$ when $h = \cosh(r)$, $h_r = \cosh(2r)$, and $v = \sinh(r)$.

Define an orthonormal basis $\{ Y_i \}_{i=1}^6$ with respect to $\mu$ by
	\begin{align}
	&Y_1 = \frac{1}{h} X_1	\hskip 50pt	Y_2 = \frac{1}{h} X_2	\hskip 50pt	Y_3 = \frac{1}{h_r} X_3  \label{eqn:orthonormal basis 1}  \\
	&Y_4 = \frac{1}{v} X_4	\hskip 50pt	Y_5 = \frac{1}{v} X_5		\hskip 50pt	Y_6 = X_6.  \nonumber
	\end{align}
Our goal is to compute formulas for the components of the $(4,0)$ curvature tensor $R_\mu$ in terms of the warping functions $h, h_r,$ and $v$ (this is the content of Theorem \ref{thm:curvature tensor 1}).  
As a first step, we need to compute the components of the $(4,0)$ curvature tensor $R_{\bf c_3}$ of the complex hyperbolic metric with respect to the orthonormal basis given above.  
We can do this with the help of formula \eqref{eqn:J curvature formula}.
To use this formula note that, by construction, we have that $JY_4 = Y_1$, $JY_5 = Y_2$, and $JY_6 = Y_3$ (again, when the metric is ${\bf c_3}$, that is, when $h = \cosh(r)$, $h_r = \cosh(2r)$, and $v = \sinh(r)$).
Lastly, we use the notation
	\begin{equation*}
	R^{\bf c_3}_{ijkl} := \la R_{\bf c_3}(Y_i,Y_j)Y_k, Y_l \ra_{\bf c_3}.
	\end{equation*}
Then, up to the symmetries of the curvature tensor, the nonzero components of the $(4,0)$ curvature tensor $R_{\bf c_3}$ are

	\begin{align}
	-4 = &R^{\bf c_3}_{1414} = R^{\bf c_3}_{2525} = R^{\bf c_3}_{3636}  \label{eqn:c3 tensor 1}  \\
	-1 = &R^{\bf c_3}_{1212} = R^{\bf c_3}_{1313} = R^{\bf c_3}_{1515} = R^{\bf c_3}_{1616} = R^{\bf c_3}_{2323} = R^{\bf c_3}_{2424}  \label{eqn:c3 tensor 2}  \\
	&= R^{\bf c_3}_{2626} = R^{\bf c_3}_{3434} = R^{\bf c_3}_{3535} = R^{\bf c_3}_{4545} = R^{\bf c_3}_{4646} = R^{\bf c_3}_{5656}  \nonumber  \\
	-2 = &R^{\bf c_3}_{1425} = R^{\bf c_3}_{1436} = R^{\bf c_3}_{2536}  \label{eqn:c3 tensor 3}  \\
	-1 = &R^{\bf c_3}_{1245} = R^{\bf c_3}_{1346} = R^{\bf c_3}_{2356}  = R^{\bf c_3}_{1524} = R^{\bf c_3}_{1634} = R^{\bf c_3}_{2635}.  \label{eqn:c3 tensor 4}
	\end{align}

\subsection{Lie brackets}
We now need to compute the values of the Lie brackets of the orthogonal basis $\{ X_i \}_{i=1}^6$.  
A first observation is that, by construction, each of these vector fields is invariant under the flow of $\ddr$.  
This implies that $[X_i, X_6] = 0$ for all $1 \leq i \leq 6$.  
From this we can deduce that
	\begin{align*}
	[Y_1,Y_6] = \frac{h'}{h} &Y_1	 \hskip 30pt	[Y_2,Y_6] = \frac{h'}{h} Y_2	\hskip 30pt	[Y_3,Y_6] = \frac{h_r'}{h_r} Y_3  \\
		&[Y_4,Y_6] = \frac{v'}{v} Y_4	\hskip 30pt	[Y_5,Y_6] = \frac{v'}{v} Y_5.
	\end{align*}
	
Next, we know that each Lie bracket is tangent to the level surfaces of $r$.  
Thus, for all $1 \leq i, j \leq 6$, the Lie bracket $[X_i, X_j]$ has no $X_6$ term.  
For all $1 \leq i, j, k \leq 5$ define structure constants $c_{ij}^k$ by
	\begin{equation}\label{eqn:structure constants 1}
	[X_i, X_j] = \sum_{k = 1}^5 c_{ij}^k X_k.
	\end{equation}
	
Two quick observations about the structure constants.  
The first is that $c_{ij}^k = - c_{ji}^k$ due to the anti-symmetry of the Lie bracket.  
The second observation is about the values of $c_{45}^4$ and $c_{45}^5$.  
Recall the definitions for $\check{X}_4$ and $\check{X}_5$ from equation \eqref{eqn:X4 and X5 check}.
Then
	\begin{equation}\label{eqn:first Lie bracket 1}
	\left[ \check{X}_4, \check{X}_5 \right] = \left[ \ddt, \frac{1}{\sin (\th)} \dds \right] = \frac{- \cos (\th)}{\sin^2 (\th)} \dds = - \cot (\th) \check{X}_5.
	\end{equation}
We therefore conclude that $c_{45}^4 = 0$ and $c_{45}^5 = - \cot(\th)$.

The following Theorem gives almost a full description of the values of the Lie brackets.  
Some quantities are only defined up to sign, but this is sufficient to compute the curvature formulas in Theorem \ref{thm:curvature tensor 1}.
The interested reader can find the proof of Theorem \ref{thm:Lie brackets 1} in Section \ref{section:proof of Lie brackets 1}.  

\vskip 15pt

\begin{theorem}\label{thm:Lie brackets 1}
The values for the Lie brackets in equation \eqref{eqn:structure constants 1} are
	\begin{align*}
	&[X_1,X_2] =  \pm X_1  		\hskip 101pt  	[X_1,X_3] =  \mp \cot(\th) X_2 + X_4  \\
	&[X_1,X_4] =  X_3 \mp X_5  	\hskip 84pt	[X_1,X_5] =  - \cot(\th) X_2 \pm X_4  \\
	&[X_2,X_3] =  \pm \cot(\th) X_1 + X_5  	\hskip 48pt	[X_2,X_4] =  0  
	\end{align*}
	\begin{align*}
	&[X_2,X_5] =  \cot(\th) X_1 + X_3  	\hskip 58pt	[X_3,X_4] =  -X_1 \pm \cot(\th) X_5  \\
	&[X_3,X_5] =  -X_2 \mp \cot(\th) X_4  	\hskip 50pt	[X_4,X_5] =  - \cot(\th) X_5
	\end{align*}
In the above equations, all of the $\pm$ and $\mp$ signs are related.  
For example, if it is the case that $[X_1,X_2] = X_1$, then $[X_1,X_4] = X_3 - X_5$ and so on.
\end{theorem}

\subsection{The Levi-Civita connection and formulas for the (4,0) curvature tensor $R_\mu$}
In this Subsection we first compute the Levi-Civita connection $\nabla$ associated to the metric $\mu$ with respect to the frame $(Y_i)_{i=1}^6$.  	
The difficult part in all of this is computing the Lie brackets in Theorem \ref{thm:Lie brackets 1}.
From there it is now a simple calculation using formula \eqref{eqn:connection} to prove the following Theorem.

\vskip 15pt

\begin{theorem}\label{thm:connection 1}
The Levi-Civita connection $\nabla$ compatible with $\mu$ is determined by the $36$ equations
	\begin{align*}
	\bullet \nabla_{Y_1}Y_1 = & \mp \frac{1}{h} Y_2 - \frac{h'}{h} Y_6 \hskip 26pt  \bullet \nabla_{Y_3}Y_4 = - \frac{1}{2} \left( \frac{h}{h_r v} - \frac{h_r}{hv} - \frac{v}{hh_r} \right) Y_1 \pm \frac{1}{h_r} \cot(\th) Y_5  \\
	\bullet \nabla_{Y_1}Y_2 = & \pm \frac{1}{h} Y_1    \hskip 60pt	\bullet \nabla_{Y_3}Y_5 =  - \frac{1}{2} \left( \frac{h}{h_r v} - \frac{h_r}{hv} - \frac{v}{hh_r} \right) Y_2 \mp \frac{1}{h_r} \cot(\th) Y_4  \\
	\bullet \nabla_{Y_1}Y_3 = & \frac{1}{2} \left( \frac{h}{h_r v} - \frac{h_r}{hv} + \frac{v}{h h_r} \right) Y_4  \hskip 44pt	  \bullet \nabla_{Y_4}Y_1 = - \frac{1}{2} \left( \frac{h}{h_r v} + \frac{h_r}{hv} + \frac{v}{hh_r} \right) Y_3  \\
	\bullet \nabla_{Y_1}Y_4 = & - \frac{1}{2} \left( \frac{h}{h_r v} - \frac{h_r}{hv} + \frac{v}{hh_r} \right) Y_3 \mp \frac{1}{h} Y_5   \hskip 112pt  \bullet \nabla_{Y_4}Y_2 = 0  \\ %\label{eqn62}  \\
	\bullet \nabla_{Y_1}Y_5 = & \pm \frac{1}{h} Y_4	\hskip 132pt  \bullet \nabla_{Y_4}Y_3 = \frac{1}{2} \left( \frac{h}{h_r v} + \frac{h_r}{hv} + \frac{v}{hh_r} \right) Y_1  \\
	\bullet \nabla_{Y_2}Y_1 = & 0    \hskip 238pt		\bullet \nabla_{Y_4}Y_4 = - \frac{v'}{v} Y_6  \\
	\bullet \nabla_{Y_2}Y_2 = & - \frac{h'}{h} Y_6	\hskip 232pt  \bullet \nabla_{Y_4}Y_5 = 0  \\
	\bullet \nabla_{Y_2}Y_3 = & \frac{1}{2} \left( \frac{h}{h_r v} - \frac{h_r}{hv} + \frac{v}{h h_r} \right) Y_5  \hskip 117pt  \bullet \nabla_{Y_5}Y_1 = \frac{1}{v} \cot(\th) Y_2  \\ %\label{eqn63}  \\
	\bullet \nabla_{Y_2}Y_4 = & 0 	\hskip 91pt  \bullet \nabla_{Y_5}Y_2 = - \frac{1}{v} \cot(\th) Y_1 - \frac{1}{2} \left( \frac{h}{h_r v} + \frac{h_r}{hv} + \frac{v}{hh_r} \right) Y_3  \\
	\bullet \nabla_{Y_2}Y_5 = & - \frac{1}{2} \left( \frac{h}{h_r v} - \frac{h_r}{hv} + \frac{v}{h h_r} \right) Y_3 	\hskip 40pt	\bullet \nabla_{Y_5}Y_3 = \frac{1}{2} \left( \frac{h}{h_r v} + \frac{h_r}{hv} + \frac{v}{hh_r} \right) Y_2  \\
	\bullet \nabla_{Y_3}Y_1 = & \pm \frac{1}{h_r} \cot(\th) Y_2 + \frac{1}{2} \left( \frac{h}{h_r v} - \frac{h_r}{vh} - \frac{v}{hh_r} \right) Y_4	\hskip 40pt  \bullet \nabla_{Y_5}Y_4 = \frac{1}{v} \cot(\th) Y_5  \\
	\bullet \nabla_{Y_3}Y_2 = & \mp \frac{1}{h_r} \cot(\th) Y_1 + \frac{1}{2} \left( \frac{h}{h_r v} - \frac{h_r}{hv} - \frac{v}{hh_r} \right) Y_5    \\
	\bullet \nabla_{Y_3}Y_3 = & - \frac{h_r'}{h_r} Y_6	\hskip 148pt  \bullet \nabla_{Y_5}Y_5 = - \frac{1}{v} \cot(\th) Y_4 - \frac{v'}{v} Y_6	  \\ %\label{eqn65}  \\
	\bullet \nabla_{Y_1}Y_6 = &\frac{h'}{h} Y_1	\hskip 3pt	\bullet \nabla_{Y_2}Y_6 = \frac{h'}{h} Y_2	\hskip 3pt	\bullet \nabla_{Y_3}Y_6 = \frac{h_r'}{h_r}Y_3	\hskip 3pt	\bullet \nabla_{Y_4}Y_6 = \frac{v'}{v} Y_4	\hskip 3pt	\bullet \nabla_{Y_5}Y_6 = \frac{v'}{v} Y_5  \\
	\bullet \, &0 = \nabla_{Y_6}Y_1 = \nabla_{Y_6}Y_2 = \nabla_{Y_6}Y_3 = \nabla_{Y_6}Y_4 = \nabla_{Y_6}Y_5 = \nabla_{Y_6}Y_6 %\label{eqn611} 
	\end{align*}
\end{theorem}

\vskip 10pt

By combining Theorem \ref{thm:connection 1} with equation \eqref{eqn:curvature tensor notation}, and remembering that $Y_6 = X_6 = \ddr$ and $X_4 = \ddt$, we compute the following formulas for the $(4,0)$ curvature tensor $R_\mu$.  
As in equations \eqref{eqn:c3 tensor 1} through \eqref{eqn:c3 tensor 4} we use the notation:
	\begin{equation*}
	R^{\mu}_{ijkl} := \la R_{\mu}(Y_i,Y_j)Y_k, Y_l \ra_\mu.
	\end{equation*}

\vskip 20pt

\begin{theorem}\label{thm:curvature tensor 1}
In terms of the basis given in equation \eqref{eqn:orthonormal basis 1}, the only independent nonzero components of the $(4,0)$ curvature tensor $R_{\mu}$ are the following:
	\begin{align*}
	&R^{\mu}_{1212} = - \left( \frac{h'}{h} \right)^2 - \frac{1}{h^2}  \hskip  30pt  R^{\mu}_{4545} = - \left( \frac{v'}{v} \right)^2 + \frac{1}{v^2}  \hskip 30pt  R^{\mu}_{1515} = R^{\mu}_{2424} = - \frac{h' v'}{hv}  \\
	&R^{\mu}_{1414} = R^{\mu}_{2525} = - \frac{v'h'}{v h} - \left( \frac{-v^2}{4h^2h_r^2} - \frac{h^2}{4v^2h_r^2} + \frac{3h_r^2}{4v^2h^2} - \frac{1}{2v^2} + \frac{1}{2h^2} - \frac{1}{2h_r^2} \right)  \\
	&R^{\mu}_{3434} = R^{\mu}_{3535} = - \frac{v'h_r'}{vh_r} - \left( \frac{- v^2}{4h^2h_r^2} + \frac{3h^2}{4v^2h_r^2} - \frac{h_r^2}{4v^2h^2} - \frac{1}{2v^2} - \frac{1}{2h^2} + \frac{1}{2h_r^2}  \right) \\
	&R^{\mu}_{1313} = R^{\mu}_{2323} = - \frac{h' h_r'}{h h_r} - \left( \frac{3v^2}{4h^2 h_r^2} - \frac{h^2}{4v^2h_r^2} - \frac{h_r^2}{4v^2h^2} + \frac{1}{2v^2} + \frac{1}{2h^2} + \frac{1}{2h_r^2} \right)  \\
	&R^{\mu}_{1616} = R^{\mu}_{2626} = - \frac{h''}{h} \hskip 50pt R^{\mu}_{3636} = - \frac{h_r''}{h_r}  \hskip 50pt  R^{\mu}_{4646} = R^{\mu}_{5656} = - \frac{v''}{v}  \\
	&R^{\mu}_{1436} = R^{\mu}_{2536} = \frac{1}{2h_r} \left[ \left( \frac{h}{v} \right)' - \left( \frac{v}{h} \right)' - \left( \frac{h_r^2}{vh} \right)' \right]  \\
	&R^{\mu}_{1634} = R^{\mu}_{2635} = \frac{1}{2h} \left[ - \left( \frac{h_r}{v} \right)' + \left( \frac{v}{h_r} \right)' + \left( \frac{h^2}{vh_r} \right)' \right] \\
	&R^{\mu}_{1346} = R^{\mu}_{2356} = \frac{-1}{2v} \left[ \left( \frac{h}{h_r} \right)' + \left( \frac{h_r}{h} \right)' + \left( \frac{v^2}{h h_r} \right)' \right] \\
	&R^{\mu}_{1425} =  \frac{1}{2v^2} - \frac{1}{2h^2} - \frac{h_r^2}{2 h^2 v^2}  \\
	&R^{\mu}_{1245} =  - \frac{1}{4} \left( \frac{h^2}{h_r^2 v^2} + \frac{h_r^2}{h^2 v^2} + \frac{v^2}{h^2 h_r^2} + \frac{2}{h^2} + \frac{2}{h_r^2} - \frac{2}{v^2} \right)  \\
	&R^{\mu}_{1524} =  - \frac{1}{4} \left( \frac{-h^2}{h_r^2 v^2} + \frac{h_r^2}{h^2 v^2} - \frac{v^2}{h^2 h_r^2} - \frac{2}{h_r^2} \right).
	\end{align*}
\end{theorem}
	
\vskip 10pt
	
It is a tedious exercise in hyperbolic trigonometric identities to check that, when $h = \cosh(r)$, $h_r = \cosh(2r)$, and $v = \sinh(r)$, the above formulas reduce to the constants in equations \eqref{eqn:c3 tensor 1} through \eqref{eqn:c3 tensor 4}.  
Also, note that the first nine equations above give the sectional curvatures of the coordinate planes, while the last six equations are formulas for the nonzero mixed terms.
%Also, note that the curvature tensor on a four dimensional Riemannian manifold has $21$ components which are independent with respect to the symmetries of the curvature tensor.  
%So Theorem \ref{thm:curvature tensor} also states that the remaining $12$ components of $\Rl$ are identically zero.

Finally, notice that the above curvature formulas contain all of the formulas that arise in the analogous $\C \H^n \setminus \H^n$ for general $n$.  
In general, one can write the complex hyperbolic metric ${\bf c_n}$ as
	\begin{equation*}
	{\bf c_n} = \cosh^2(r) {\bf h_{n-1}} + \cosh^2(2r) dX_n^2 + \sinh^2(r) {\bf \s_{n-1}} + dr^2
	\end{equation*}
and the corresponding warped-product metric as
	\begin{equation*}
	{\bf \mu_n} = h^2(r) {\bf h_{n-1}} + h_r^2(r) d X_n^2 + v^2(r) {\bf \s_{n-1}} + dr^2
	\end{equation*}
where ${\bf \s_{n-1}}$ is the round metric on $\S^{n-1}$ and the vector field $X_n$ is defined in the same manner as $X_3$.  
The curvature formulas for the base $\H^{n}$ are encoded in the formulas for $R^{\mu}_{1212}$ and $R^{\mu}_{1313}$.  
All curvature formulas for the $\S^{n-1}$ factor are contained in the term $R^{\mu}_{4545}$.
Note that neither of these cases contain any mixed terms.
Adding in the curvature formulas above of the form $R^{\mu}_{1414}$, $R^{\mu}_{3434}$, $R^{\mu}_{1425}$, $R^{\mu}_{1245}$, and $R^{\mu}_{1524}$ gives all curvature formulas for $h^2(r) {\bf h_{n-1}} + h_r^2(r) dX_n^2 + v^2(r) {\bf \s_{n-1}}$ (this is where most of the mixed terms appear).  
And then all of the formulas above containing a ``6" give the rest of the curvature formulas for $\mu_n$.

\vskip 20pt

%beginning of section 4
\section{Curvature formulas for warped product metrics on $\C \H^n \setminus \C \H^k$}\label{section:chn/chk}

As mentioned in the Introduction, for simplicity we are going to restrict ourselves to the case when $n=5$ and $k=2$.  
These are the smallest choices for $n$ and $k$ which capture every formula for the curvature tensor in the general case, so nothing is lost with this restriction (see the comments after Theorem \ref{thm:curvature tensor 2} for more discussion).

\subsection{Expressing the metric in $\C \H^5$ in spherical coordinates about $\C \H^2$}
Let ${\bf c_5}$ denote the complex hyperbolic metric on $\C \H^5$ normalized to have constant holomorphic sectional curvature $-4$.  
Since $\C \H^2$ is a complete totally geodesic submanifold of the negatively curved manifold $\C \H^5$, there exists an orthogonal projection map $\pi : \C \H^5 \to \C \H^2$.  
This map $\pi$ is a fiber bundle whose fibers are totally geodesic $6$-planes isometric to $\C \H^3$.  

For $r > 0$ let $E(r)$ denote the $r$-neighborhood of $\C \H^2$.  
Then $E(r)$ is a real hypersurface in $\C \H^5$, and consequently we can decompose ${\bf c_5}$ as
	\begin{equation*}
	{\bf c_5} = ({\bf c_5})_r + dr^2
	\end{equation*}
where $({\bf c_5})_r$ is the induced Riemannian metric on $E(r)$.  
Let $\pi_r : E(r) \to \C \H^2$ denote the restriction of $\pi$ to $E(r)$.  
Note that $\pi_r$ is an $\S^5$-bundle whose fiber over any point $q \in \C \H^2$ is (topologically) the $5$-sphere of radius $r$ in the totally geodesic $6$-plane $\pi^{-1}(q)$.  
The tangent bundle splits as an orthogonal sum $\mathcal{V}(r) \oplus \mathcal{H}(r)$ where $\mathcal{V}(r)$ is tangent to the $5$-sphere $\pi_r^{-1}(q)$ and $\mathcal{H}(r)$ is the orthogonal complement to $\mathcal{V}(r)$.  
Note that this copy of $\S^5$ does {\em not} have constant sectional curvature equal to $1$, but rather it is an example of a {\em Berger sphere}.  
This will be discussed further below.

For $r, s > 0$ there exists a diffeomorphism $\phi_{sr}: E(s) \to E(r)$ induced by the geodesic flow along the totally geodesic $6$-planes orthogonal to $\C \H^2$.   
Fix $p \in E(r)$ arbitrary, let $q = \pi(p) \in \C \H^2$, and let $\g$ be the unit speed geodesic such that $\g(0) = q$ and $\g(r) = p$. 
In what follows, all computations are considered in the tangent space $T_p E(r)$. 

Note that $\V(r)$ is tangent to both $E(r)$ and the totally geodesic $6$-plane $\pi^{-1}(q)$.  
Then since $\pi^{-1}(q)$ is preserved by the geodesic flow, we have that $d \phi_{sr}$ takes $\V(s)$ to $\V(r)$.  
Consider the complex geodesic $P = exp(span(\ddr, J \ddr))$. 
P intersects $E(r)$ orthogonally, and $P \cap E(r)$ is isometric to a circle of radius $r$.  
Thus, since a complex geodesic has curvature $-4$, there exists a suitable identification $P \cong \S^1 \times (0,\infty)$ where the metric ${\bf c_5}$ restricted to $P$ can be written as
	\begin{equation*}
	\frac{1}{4} \sinh^2(2r) d \th^2 + dr^2
	\end{equation*}
where $d \th^2$ denotes the round metric on the unit circle $\S^1$.  
Note that the presence of the $``1/4"$ is to make the metric complete when extended to the core $\C \H^2$.

Notice that $\ddt$ is a vector field on the five sphere $\S^5$ mentioned above.  
More generally, thinking of $\S^5$ as the unit sphere in $\C^3$ with respect to the usual Hermitian metric, there is an obvious free action of the circle $\S^1$ on $\S^5$.  
The unit tangent vector field with respect to this action corresponds to the vector field $\ddt$ above.
This action fibers $\S^5$ over the complex projective plane $\C \P^2$, and the Riemannian submersion metric on this fiber bundle is an example of a Berger sphere (see \cite{Farrell Jones} pg.\ 59 for more details).  
Let $\a(t)$ be a unit speed geodesic in $\S^5$ orthogonal to $J \ddr$ such that $\a(0) = p$.
%$\C \P^2$ such that $\a(0) = p$ (and where we are considering the identification $E(r) \cong \C \H^2 \times \C \P^2 \times \S^1$).
Then $\exp_p(\a'(0), \partial r)$ forms a totally real totally geodesic 2-plane in $\C \H^5$.  
Thus the curvature of this $2$-plane is $-1$.  
Since the direction of $\a$ orthogonal to $J \ddr$ was arbitrary, we can write the Riemannian metric $({\bf c_5})_r$ restricted to $\mathcal{V}(r)$ as
	\begin{equation*}
	\sinh^2(r) {\bf p_2} + \frac{1}{4} \sinh^2(2r) d \th^2
	\end{equation*}
where ${\bf p_2}$ denotes the complex projective metric on $\C \P^2$.  

%%%%%%%%%%%%%%%%%%%%%%%%%%%%%%%%%%%%%%%%%%%%%%%%%%%%%%%%%%%%%%%%%%
\begin{comment}
More generally, $exp_p^{-1} \left( \pi^{-1}(q) \right)$ is a complex $3$-plane.  
So $exp^{-1}(P)^{\perp} \cap exp^{-1}(\pi^{-1}(q))$ is isometric to $\C \H^2$.
Since every vector $v \in exp_p^{-1} \left( \pi^{-1}(q) \right)$ which is orthogonal to $exp^{-1}(P)$ is antiholomorphic to $\ddr$, we have that $K(v,\ddr) = -1$.  
Therefore, there exists a suitable identification $\pi^{-1}(q) \cong \C \H^1 \times \C \H^1 \times \S^1 \times (0,\infty)$ where the metric ${\bf c_4}$ restricted to $\pi^{-1}(q)$ can be written as 
	\begin{equation*}
	\sinh^2(r) {\bf c_1} + \sinh^2(r) {\bf c_1} + \sinh^2(2r) \s^1 + dr^2.
	\end{equation*}
Here, ${\bf c_1}$ is the standard metric on the complex hyperbolic line $\C \H^1$.
%Note that the reason for the two copies of $\S^2$ is because the orthogonal complement to $\exp^{-1}(P)$ is isometric to $\C \H^2$.
Note that the metric above is obviously not the product metric. %as the two different $\S^2$ distributions are not integrable.
\end{comment}
%%%%%%%%%%%%%%%%%%%%%%%%%%%%%%%%%%%%%%%%%%%%%%%%%%%%%%%%%%%%%%%%%%

Now let $\b(t)$ be any unit speed geodesic in $\C \H^2$ such that $\b(0) = q$.  
Then $Q = exp(span(\b'(0), \g'(0)))$ is a totally real totally geodesic submanifold of $\C \H^5$, and thus $K(\b', \g') = -1$.  
Therefore, the metric ${\bf c_5}$ restricted to $Q$ can be written as $\cosh^2(r) dt^2 + dr^2$.  
But since $\g$ was arbitrary, we can write the metric on the 5-dimensional submanifold determined by $\C \H^2$ and $\ddr$ as $\cosh^2 {\bf c_2} + dr^2$.
This leads to the following.

\vskip 20pt

\begin{theorem}\label{thm:chn/chk metric}
The complex hyperbolic manifold $\C \H^5 \setminus \C \H^2$ can be written as $E \times (0, \infty)$  where $E \cong \C \H^2 \times \S^5$ equipped with the metric
	\begin{equation}\label{eqn:chn/chk metric}
	{\bf c_5} = \cosh^2(r) {\bf c_2} + \sinh^2(r) {\bf p_2} + \frac{1}{4} \sinh^2(2r) d \th^2 + dr^2 .
	\end{equation}
\end{theorem}

\vskip 10pt

\subsection{The warped product metric, orthonormal basis, and curvature formulas in $\C \H^5 \setminus \C \H^2$}
For some positive, increasing real-valued functions $h, v, v_r : (0,\infty) \to \R$ define the Riemannian metrics
	\begin{equation*}
	{\bf \g_{r,\th}} = h^2(r) {\bf c_2} + v^2(r) {\bf p_2}
	\end{equation*}
	\begin{equation*}
	{\bf \g_r} := {\bf \g_{r,\th}} + \frac{1}{4} v_r^2(r) d \th^2    
	\end{equation*}
and
	\begin{equation}\label{eqn:/g}
	{\bf \g} := {\bf \g_r} + dr^2.
	\end{equation}
Of course, ${\bf \g} = {\bf c_5}$ when $h = \cosh(r)$, $v = \sinh(r)$, and $v_r = \sinh(2r)$.

For the remainder of this Section, fix $p = (q_1,\bar{q},r) \in \C \H^2 \times \S^5 \times (0,\infty) \cong \C \H^5 \setminus \C \H^2$, and write $\bar{q} \in \S^5$ as $(q_2, \th)$ where $q_2 \in \C \P^2$ and $\th \in \S^1$.  
Let $(\check{X}_1,\check{X}_2, \check{X}_3, \check{X}_4)$ be an orthonormal collection of vector fields near $q_1 \in \C \H^2$ which satisfies:
	\begin{enumerate}
	 \item  $[\check{X}_i,\check{X}_j]_{q_1} = 0$ for all $1 \leq i, j \leq 4$.
	\item  $J \check{X}_2|_{q_1} = \check{X}_1|_{q_1}$ and $J \check{X}_4|_{q_1} = \check{X}_3|_{q_1}$.
	\end{enumerate}
Define an analogous collection of vector fields $(\check{X}_5, \check{X}_6, \check{X}_7, \check{X}_8)$ about $q_2 \in \C \P^2$ so that $J \check{X}_6|_{q_2} = \check{X}_5|_{q_2}$,  $J \check{X}_8|_{q_2} = \check{X}_7|_{q_2}$, and $[\check{X}_i,\check{X}_j]_{q_2} = 0$ for all $5 \leq i, j \leq 8$.  
Extend both collections to vector fields $(X_1, \hdots, X_8)$ near $p$. 
%via the inclusions $\C \H^2, \C \P^2 \to \C \H^5 \setminus \C \H^2$.  
Lastly, let $X_9 = \ddt$ and $X_{10} = \ddr$.  

Define an orthonormal basis $\{ Y_i \}_{i=1}^8$ with respect to $\g$ by
	\begin{align}
	&Y_1 = \frac{1}{h} X_1	\hskip 25pt	Y_2 = \frac{1}{h} X_2	\hskip 25pt	Y_3 = \frac{1}{h} X_3		\hskip 25pt	Y_4 = \frac{1}{h} X_4  \hskip 25pt  Y_5 = \frac{1}{v} X_5  \label{eqn:orthonormal basis 2}	  \\
	&Y_6 = \frac{1}{v} X_6		\hskip 25pt	Y_7 = \frac{1}{v} X_7		\hskip 25pt	Y_8 = \frac{1}{v} X_8	\hskip 25pt	Y_9 = \frac{1}{\frac{1}{2} v_r} X_9  \hskip 25pt  Y_{10} = X_{10}.  \nonumber
	\end{align}
	
Our goal is to compute formulas for the components of the $(4,0)$ curvature tensor $R_\g$ in terms of the warping functions $h, v,$ and $v_r$.  
As a first step, we need to compute the components of the $(4,0)$ curvature tensor $R_{\bf c_5}$ of the complex hyperbolic metric with respect to the orthonormal basis given above.  
Just as in Section \ref{section:chn/hn} we can do this with the help of formula \eqref{eqn:J curvature formula}.
To use this formula note that, by construction, we have that $JY_2 = Y_1$, $JY_4 = Y_3$, $JY_6 = Y_5$, $JY_8 = Y_7$, and $J Y_{10} = Y_9$ at the point $p$ (and again, when the metric is ${\bf c_5}$.  So, when $h = \cosh(r)$, $v = \sinh(r)$, and $v_r = \sinh(2r)$).
Lastly, we use the notation
	\begin{equation*}
	R^{\bf c_5}_{ijkl} := \la R_{\bf c_5}(Y_i,Y_j)Y_k, Y_l \ra_{\bf c_5}.
	\end{equation*}
Then, up to the symmetries of the curvature tensor, the nonzero components of the $(4,0)$ curvature tensor $R_{\bf c_5}$ are

	\begin{align} 
	-4 = &R^{\bf c_5}_{1212} = R^{\bf c_5}_{3434} = R^{\bf c_5}_{5656} = R^{\bf c_5}_{7878} = R^{\bf c_5}_{9,10,9,10}  \label{eqn:c4 tensor 1}  \\
	-1 = &R^{\bf c_5}_{ijij} \text{ where } \{ i,j \} \nin \{ \{1,2 \}, \{ 3, 4 \}, \{ 5, 6 \}, \{ 7, 8 \}, \{ 9, 10 \} \}  \label{eqn:c4 tensor 2}  \\
	%= R^{\bf c_4}_{1414} = R^{\bf c_4}_{1515} = R^{\bf c_4}_{1616} = R^{\bf c_4}_{1717} = R^{\bf c_4}_{1818} = R^{\bf c_4}_{2323} = R^{\bf c_4}_{2424} \label{eqn:c4 tensor 2}  \\
	%&= R^{\bf c_4}_{2525} = R^{\bf c_4}_{2626} = R^{\bf c_4}_{2727} = R^{\bf c_4}_{2828} = R^{\bf c_4}_{3535} = R^{\bf c_4}_{3636} = R^{\bf c_4}_{3737} = R^{\bf c_4}_{3838}  \nonumber  \\
	%&= R^{\bf c_4}_{4545} = R^{\bf c_4}_{4646} = R^{\bf c_4}_{4747} = R^{\bf c_4}_{4848} = R^{\bf c_4}_{5757} = R^{\bf c_4}_{5858} = R^{\bf c_4}_{6767} = R^{\bf c_4}_{6868}  \nonumber  \\
	-2 = &R^{\bf c_5}_{ijkl} \text{ where }  (i,j) \neq (k, l) \in \{ (1,2), ( 3,4), (5,6), (7,8), (9,10) \}  \label{eqn:c4 tensor 3}  \\  %\text{ and } \{ i, j \} \neq \{ k, l \}   \\
	%= R^{\bf c_4}_{1256} = R^{\bf c_4}_{1278} = R^{\bf c_4}_{3456} = R^{\bf c_4}_{3478} = R^{\bf c_4}_{5678}  \label{eqn:c4 tensor 3}  \\
	-1 = &R^{\bf c_5}_{ikjl} \text{ where }  (i,j) \neq (k, l) \in \{ (1,2), ( 3, 4), ( 5, 6), ( 7, 8), ( 9, 10) \}    \label{eqn:c4 tensor 4}  \\
	%= R^{\bf c_4}_{1526} = R^{\bf c_4}_{1728}  = R^{\bf c_4}_{3546} = R^{\bf c_4}_{3748} = R^{\bf c_4}_{5768}.  \label{eqn:c4 tensor 4}  \\
	1 = &R^{\bf c_5}_{iljk} \text{ where }  (i,j) \neq (k, l) \in \{ (1,2), ( 3, 4), ( 5, 6), ( 7, 8), ( 9, 10) \}  \label{eqn:c4 tensor 5}
	%= R^{\bf c_4}_{1625} = R^{\bf c_4}_{1827}  = R^{\bf c_4}_{3645} = R^{\bf c_4}_{3847} = R^{\bf c_4}_{5867}.  \label{eqn:c4 tensor 5}
	\end{align}
	
Let us quickly note that, since $\C \P^2$ is dual to $\C \H^2$, we have the following curvature formulas for $R_{\bf p_2}$:
	\begin{align*}
	4 &= R^{\bf p_2}_{5656} = R^{\bf p_2}_{7878}  \\
	1 &= R^{\bf p_2}_{5757} = R^{\bf p_2}_{5858} = R^{\bf p_2}_{6767} = R^{\bf p_2}_{6868}  \\
	2 &= R^{\bf p_2}_{5678} = 2 R^{\bf p_2}_{5768} = -2 R^{\bf p_2}_{5867}.
	\end{align*}
In the above formulas, $R^{\bf p_2}_{ijkl} := \la R_{\bf p_2}(Y_i,Y_j)Y_k,Y_l \ra_{\bf p_2}$ and with the abuse of notation of $Y_i$ denoting the restriction of $Y_i$ to $\C \P^2$.

\subsection{Lie brackets and curvature formulas for ${\bf \g_{r, \th}}$}
The vector fields $\{ X_i \}_{i=1}^{10}$ form an orthogonal frame near $p$ which satisfies the following properties (at $p$):
	\begin{enumerate}
	\item $[X_i,X_j]$ is tangent to the level surfaces of $r$ for $1 \leq i, j \leq 9$.
	\item $[X_i,X_j]$ is tangent to $\C \H^2 \times \S^1$ for $1 \leq i, j \leq 4$, where $\S^1 \cong \text{exp}_p(J \ddr)$.
	\item $[X_i,X_j]$ is tangent to $\S^5$ for $5 \leq i, j \leq 8$.
	\item $[X_i,X_{10}] = 0$ since $X_i$ is invariant under the flow of $\ddr$ for $1 \leq i \leq 9$.
	\item $[X_i,X_9] = 0$ since $X_i$ is invariant under the flow of $\ddt$ for $1 \leq i \leq 8$.
	\item $[X_i, X_j] = 0$ for $i \in \{ 1, 2, 3, 4 \}$ and $j \in \{5, 6, 7, 8 \}$ since these vector fields were defined via inclusion.
	\end{enumerate}
By the above points, and since $[\check{X}_i,\check{X}_j]_p = 0$ for all $1 \leq i, j \leq 8$, there exist {\em structure constants} $c_{ij}$ such that $[X_i,X_j]_p = c_{ij} X_9$.
Note that $c_{ij} = - c_{ji}$.  
The following Lemma provides the values for the structure constants.

\vskip 15pt

\begin{lemma}\label{lemma:structure constants}
The values for the structure constants are $c_{12} = c_{34} = c_{56} = c_{78} = 2$, and all other (independent) structure constants are equal to zero.
\end{lemma}

\vskip 10pt

A quick note is that by ``independent" structure constants we just mean that, obviously, $c_{21} = c_{43} = c_{65} = c_{87} = - 2 \neq 0$.

\begin{proof}
All of the structure constants can be found by combining formula \eqref{eqn:Bel 4} with the curvature formulas \eqref{eqn:c4 tensor 3} through \eqref{eqn:c4 tensor 5}.  
To see that $c_{12}= 2$, we combine equation \eqref{eqn:c4 tensor 3} with \eqref{eqn:Bel 4} to obtain
	\begin{align*}
	4 = 2R^{\bf c_5}_{10,9,1,2} &= 0 + 0 + \la [Y_1,Y_2],Y_9 \ra_{\bf c_5} \left( \ln \left[ \frac{\frac{1}{4} \sinh^2(2r)}{\cosh^2(r)} \right] \right)'  \\
	&= \frac{c_{12} \sinh(2r)}{\cosh^2(r)} \left( \frac{\cosh(r)}{\sinh(r)} \right)  = 2 c_{12}.
	\end{align*}
An analogous argument shows that $c_{34} = c_{56} = c_{78} = 2$.  
To see that $c_{13} = c_{14} = c_{23} = c_{24} = c_{57} = c_{58} = c_{67} = c_{68} = 0$ we use the same equations as above, but note that the left hand side is now $0$ instead of $4$.

Lastly, to see that $c_{15} = 0$, note that
	\begin{align*}
	0 = R^{\bf c_5}_{10,1,5,9} &= 0 + \la [Y_5,Y_1],Y_9 \ra_{\bf c_5} \left( \ln \left[ \frac{\frac{1}{2} \sinh(2r)}{\sinh(r)} \right] \right)' + 0  \\
	&= \frac{- \frac{1}{2} c_{15} \sinh(2r)}{\sinh(r) \cosh(r)} ( \ln 2 \cosh(r) )' = - c_{15} \tanh(r).
	\end{align*}
The argument that the remaining structure constants are $0$ is identical to the argument above.
\end{proof}

We now, for some fixed $r$ and $\th$, compute the components of the $(4,0)$ curvature tensor $R_{\bf \l_{r, \th}}$ with respect to the orthonormal frame $\{ Y_i \}_{i=1}^8$.  
Since $[X_i,X_j] = 0$ for $i \in \{ 1, 2, 3, 4 \}$ and $j \in \{ 5, 6, 7, 8 \}$, the metric ${\bf \g_{r, \th}}$ is a product metric.  
Then since the $(4,0)$ curvature tensor scales like the metric, up to the symmetries of the curvature tensor the only nonzero components of $R_{\bf \g_{r, \th}}$ are
	\begin{align*}
	&R^{\bf \g_{r, \th}}_{1212} = R^{\bf \g_{r, \th}}_{3434} = - \frac{4}{h^2}  \hskip 68pt  R^{\bf \g_{r, \th}}_{1313} = R^{\bf \g_{r, \th}}_{1414} = R^{\bf \g_{r, \th}}_{2323} = R^{\bf \g_{r, \th}}_{2424} = - \frac{1}{h^2}  \\
	&  R^{\bf \g_{r, \th}}_{5656} = R^{\bf \g_{r, \th}}_{7878} = \frac{4}{v^2}  \hskip 84pt  R^{\bf \g_{r, \th}}_{5757} = R^{\bf \g_{r, \th}}_{5858} = R^{\bf \g_{r, \th}}_{6767} = R^{\bf \g_{r, \th}}_{6868} = \frac{1}{v^2}  \\
	&R^{\bf \g_{r, \th}}_{1234} = 2 R^{\bf \g_{r, \th}}_{1324} = -2 R^{\bf \g_{r, \th}}_{1423} = - \frac{2}{h^2}  \hskip 40pt  R^{\bf \g_{r, \th}}_{5678} = 2 R^{\bf \g_{r, \th}}_{5768} = -2 R^{\bf \g_{r, \th}}_{5867} = \frac{2}{v^2}.
	\end{align*}
In particular, note that mixed terms of the form $R^{\bf \g_{r, \th}}_{1256}$ are $0$.

\subsection{Curvature formulas for ${\bf \g_r}$}
Formulas \eqref{eqn:Bel 1} through \eqref{eqn:Bel 4} allow us to compute the $(4,0)$ curvature tensor $R_{\g}$ in terms of $R_{\g_r}$.  
We use a very different approach from Section \ref{section:chn/hn} to compute the nonzero components of $R_{\g_r}$.
The background for our current computations can be found in \cite{Besse} pg.\ 235-242.  
The metric $\g_r$ is a Riemannian submersion metric with (horizontal) base $\g_{r, \th}$ and (vertical) fiber $\frac{1}{4} v_r^2 d \th^2$.  
So our approach is to compute the A and T tensors of $\g_r$, and to then use Theorem 9.28 in \cite{Besse} to compute the components of $R_{\g_r}$.  

First, the T tensor is identically zero since the vertical $\S^1$-fibers are totally geodesic.  
The argument for this is identical to that in Section 6 of \cite{Belegradek complex}.

We now compute the $A$ tensor associated with $\g_r$.
By (\cite{Besse} Prop.\ 9.24) we have that
	\begin{equation*}
	A_{X_1}X_2 = \frac{1}{2} \mathscr{V} [X_1,X_2] =  X_9.
	\end{equation*}
Analogously, $A_{X_3}X_4 = A_{X_5}X_6 = A_{X_7}X_8 =  X_9$ and $A_{X_i}X_j = 0$ if $\{ i, j \} \nin \{ \{1, 2 \}, \{ 3, 4 \}, \{5, 6\}, \{7, 8 \} \}$.  
Also, by (\cite{Besse} eqn.\ 9.21a) we have $A_{X_9}X_i = 0$ for $1 \leq i \leq 8$.  

Now, by (\cite{Besse} eqn.\ 9.21d) we see that
	\begin{equation*}
	\la A_{X_1}X_9, X_2 \ra_{\g_r} = - \la A_{X_1}X_2, X_9 \ra_{\g_r} = - \frac{1}{4} v_r^2.
	\end{equation*}
By this same equation we know that there are no other nonzero components of $A_{X_1}X_9$.  
Therefore, $A_{X_1}X_9 = - \frac{1}{4} \frac{v_r^2}{h^2} X_2$.  
Analogously, we have that
	\begin{align*}
	&A_{X_2}X_9 = \frac{1}{4} \frac{v_r^2}{h^2} X_1	\hskip 25pt	A_{X_3}X_9 = - \frac{1}{4} \frac{v_r^2}{h^2} X_4	\hskip 25pt	A_{X_4}X_9 = \frac{1}{4} \frac{v_r^2}{h^2} X_3  \\
	A_{X_5}X_9 = &- \frac{1}{4} \frac{v_r^2}{v^2} X_6	\hskip 15pt	A_{X_6}X_9 = \frac{1}{4} \frac{v_r^2}{v^2} X_5 	\hskip 15pt	A_{X_7}X_9 = - \frac{1}{4} \frac{v_r^2}{v^2} X_8	\hskip 15pt	A_{X_8}X_9 = \frac{1}{4} \frac{v_r^2}{v^2} X_7	
	\end{align*}
	
We are now ready to use Theorem 9.28 from \cite{Besse} to compute the nonzero components of $R_{\g_r}$.  
By (9.28c) we have that
	\begin{equation*}
	\la R_{\g_r}(X_1,X_9)X_1,X_9 \ra_{\g_r} = \la A_{X_1}X_9,A_{X_1}X_9 \ra_{\g_r} = \frac{1}{16} \frac{v_r^4}{h^2}
	\end{equation*}
and thus
	\begin{equation}\label{eqn:first eqn}
	R^{\bf \g_r}_{1919} = \frac{4}{h^2 v_r^2} \la R_{\g_r}(X_1,X_9)X_1,X_9 \ra_{\g_r} = \frac{v_r^2}{4 h^4}.
	\end{equation}
Identically, $R^{\bf \g_r}_{2929} = R^{\bf \g_r}_{3939} = R^{\bf \g_r}_{4949} = \frac{v_r^2}{4h^4}$.  
Also, a completely analogous computation shows that 
	\begin{equation}
	R^{\bf \g_r}_{5959} = R^{\bf \g_r}_{6969} = R^{\bf \g_r}_{7979} = R^{\bf \g_r}_{8989} = \frac{v_r^2}{4v^4}.
	\end{equation}  
	
By (9.28f) we have that
	\begin{align*}
	\la R_{\g_r}(X_1,X_2)X_1,X_2 \ra_{\g_r} &= \la R_{\g_{r, \th}}(X_1,X_2)X_1,X_2 \ra_{\g_{r, \th}} - 3 \la A_{X_1}X_2,A_{X_1}X_2 \ra_{\g_r}  \\
	&= \la R_{\g_{r, \th}}(X_1,X_2)X_1,X_2 \ra_{\g_{r, \th}} - \frac{3}{4} v_r^2
	\end{align*}
and thus
	\begin{equation}
	R^{\g_r}_{1212} = R^{\g_{r, \th}}_{1212} - \frac{3 v_r^2}{4h^4} = - \frac{4}{h^2} - \frac{3v_r^2}{4h^4} = R^{\g_r}_{3434}.
	\end{equation}
An identical argument shows that
	\begin{align}
	R^{\g_r}_{5656} = R^{\g_r}_{7878} = \frac{4}{v^2} - \frac{3v_r^2}{4v^4}.
	\end{align}
Since $A_{X_i}X_j = 0$ if $\{ i, j \} \nin \{ \{1, 2 \}, \{ 3, 4 \}, \{5, 6\}, \{7, 8 \} \}$ the above argument also provides
	\begin{equation}
	R^{\g_r}_{ijij} = R^{\g_{r, \th}}_{ijij}  \hskip 30pt \text{ if } \hskip 10pt  \{ i, j \} \nin \{ \{1, 2 \}, \{ 3, 4 \}, \{5, 6\}, \{7, 8 \} \}.
	\end{equation}
	
We now compute the mixed terms of $R_{\g_r}$.  
	\begin{align*}
	\la R_{\g_r}(X_1,X_2)X_3,X_4 \ra_{\g_r} &= \la R_{\g_{r, \th}}(X_1,X_2)X_3,X_4 \ra_{\g_{r, \th}} - 2 \la A_{X_1}X_2,A_{X_3}X_4 \ra_{\g_r}  \\
	&= \la R_{\g_{r, \th}}(X_1,X_2)X_3,X_4 \ra_{\g_{r, \th}} - \frac{1}{2} v_r^2
	\end{align*}
and therefore
	\begin{equation}
	R^{\g_r}_{1234} = R^{\g_{r, \th}}_{1234} - \frac{v_r^2}{2h^4} = - \frac{2}{h^2} - \frac{v_r^2}{2h^4} = 2 R^{\g_r}_{1324} = -2 R^{\g_r}_{1423}.
	\end{equation}
Identically
	\begin{equation}
	R^{\g_r}_{5678} = R^{\g_{r, \th}}_{5678} - \frac{v_r^2}{2v^4} = \frac{2}{v^2} - \frac{v_r^2}{2v^4} = 2 R^{\g_r}_{5768} = -2 R^{\g_r}_{5867}.
	\end{equation}
Now
	\begin{align*}
	\la R_{\g_r}(X_1,X_2)X_5,X_6 \ra_{\g_r} &= \la R_{\g_{r, \th}}(X_1,X_2)X_5,X_6 \ra_{\g_{r, \th}} - 2 \la A_{X_1}X_2,A_{X_5}X_6 \ra_{\g_r}  = - \frac{1}{2} v_r^2
	\end{align*}
and hence
	\begin{equation}
	R^{\g_r}_{1256} = - \frac{v_r^2}{2 h^2 v^2} = R^{\g_r}_{1278} = R^{\g_r}_{3456} = R^{\g_r}_{3478}.
	\end{equation}
Finally, the same argument yields
	\begin{align}
	&R^{\g_r}_{1526} = R^{\g_r}_{1728} = R^{\g_r}_{3546} = R^{\g_r}_{3748} = - \frac{v_r^2}{4h^4 v^4}  \\
	&R^{\g_r}_{1625} = R^{\g_r}_{1827} = R^{\g_r}_{3645} = R^{\g_r}_{3847} = \frac{v_r^2}{4h^4 v^4}.  \label{eqn:last eqn}
	\end{align}

\subsection{Curvature formulas for $\g$}	
Combining equations \eqref{eqn:first eqn} through \eqref{eqn:last eqn} with formulas \eqref{eqn:Bel 1} through \eqref{eqn:Bel 4} proves the following Theorem.

\vskip 20pt

\begin{theorem}\label{thm:curvature tensor 2}
In terms of the basis given in equation \eqref{eqn:orthonormal basis 2}, the only independent nonzero components of the $(4,0)$ curvature tensor $R_{\g}$ are given by the following formulas, where $i \in \{ 1, 2, 3, 4 \}$, $k \in \{ 5, 6, 7, 8 \}$, $( i, j) \in \{ (1, 2), ( 3, 4) \}$, and $(k, l) \in \{ (5, 6), ( 7, 8) \}$.

	\begin{align*}
	&R^{\g}_{1212} = R^{\g}_{3434} = - \left( \frac{h'}{h} \right)^2 - \frac{4}{h^2} - \frac{3 v_r^2}{4h^4}  \hskip 20pt R^{\g}_{5656} = R^{\g}_{7878} = - \left( \frac{v'}{v} \right)^2 + \frac{4}{v^2} - \frac{3v_r^2}{4 v^4}  \\
	&R^{\g}_{i9i9} = - \frac{h' v_r'}{h v_r} + \frac{v_r^2}{4 h^4}   \hskip 40pt	R^{\g}_{k9k9} = - \frac{v' v_r'}{v v_r} + \frac{v_r^2}{4 v^4}	\hskip 40pt	R^{\g}_{ikik} = - \frac{h' v'}{hv}   \\
	&R^{\g}_{1313} = R^{\g}_{1414} = R^{\g}_{2323} = R^{\g}_{2424} = - \left( \frac{h'}{h} \right)^2 - \frac{1}{h^2}  \hskip 40pt	\text{(continued on next page)}
	\end{align*}
	\begin{align*}
	&R^{\g}_{5757} = R^{\g}_{5858} = R^{\g}_{6767} = R^{\g}_{6868} = - \left( \frac{v'}{v} \right)^2 + \frac{1}{v^2}  \\
	&R^{\g}_{i,10,i,10} = - \frac{h''}{h}  \hskip 60pt	R^{\g}_{k,10,k,10} = - \frac{v''}{v} 	\hskip 60pt R^{\g}_{9,10,9,10} = - \frac{v_r''}{v_r}   \\
	&R^{\g}_{1234} = 2R^{\g}_{1324} = -2R^{\g}_{1423} = - \frac{2}{h^2} - \frac{v_r^2}{2h^4}  \\
	&R^{\g}_{5678} = 2R^{\g}_{5768} = -2R^{\g}_{5867} = \frac{2}{v^2} - \frac{v_r^2}{2v^4}  \\
	&R^{\g}_{ijkl} = 2R^{\g}_{ikjl} = -2R^{\g}_{iljk} = - \frac{v_r^2}{2 h^2 v^2}   \\
	&R^{\g}_{i,j,9, 10} = 2R^{\g}_{i,9,j,10} = -2R^{\g}_{i,10,j,9} = - \frac{v_r}{h^2} \left( \ln \frac{v_r}{h} \right)'   \\
	&R^{\g}_{k,l,9,10} = 2R^{\g}_{k,9,l,10} = -2R^{\g}_{k,10,l,9} = - \frac{v_r}{v^2} \left( \ln \frac{v_r}{v} \right)' .
	\end{align*}
\end{theorem}
	
\vskip 10pt
	
Unlike Section \ref{section:chn/hn}, this time is is a much simpler exercise in hyperbolic trigonometric identities to check that, when $h = \cosh(r)$, $v = \sinh(r)$, and $v_r = \sinh(2r)$, the above formulas reduce to the constants in equations \eqref{eqn:c4 tensor 1} through \eqref{eqn:c4 tensor 5}. 

Finally, notice that the above curvature formulas contain all of the formulas that arise in the analogous $\C \H^n \setminus \C \H^k$ for general $n$ and $k$.  
In general, one can write the complex hyperbolic metric ${\bf c_n}$ as
	\begin{equation*}
	{\bf c_n} = \cosh^2(r) {\bf c_k} + \sinh^2(r) {\bf p_{n-k-1}} + \frac{1}{4} \sinh^2(2r) d \th^2 + dr^2
	\end{equation*}
and the corresponding warped-product metric as
	\begin{equation*}
	{\bf \g_n} = h^2(r) {\bf c_k} + v^2(r) {\bf p_{n-k-1}} + \frac{1}{4} v_r^2(r) d \th^2 + dr^2
	\end{equation*}
where ${\bf p_{n-k-1}}$ is the complex projective metric on $\C \P^{n-k-1}$.  
The curvature formulas for the base $\C \H^k$ are encoded in the formulas for $R^{\g}_{1212}, R^{\g}_{1313}, $ and $R^{\g}_{1234}$.  
The analogous formulas for $\C \P^{n-k-1}$ are contained in $R^{\g}_{5656}, R^{\g}_{5757}, $ and $R^{\g}_{5678}$.
Adding in the curvature formulas above of the form $R^{\g}_{ijij}$ and $R^{\g}_{ijkl}$ gives all curvature formulas for $h^2(r) {\bf c_k} + v^2(r) {\bf p_{n-k-1}}$.  
And then all of the formulas above containing either a ``9" or a ``10" give the rest of the curvature formulas for ${\bf \g_n }$.

\subsection{The exceptional case $\C \H^n \setminus \C \H^{n-2}$}\label{subsection:chn/chn-2}
Notice that, when $k = n-2$, there are no sectional curvatures of $\g$ of the form
	\begin{equation*}
	- \left( \frac{v'}{v} \right)^2 + \frac{1}{v}^2.
	\end{equation*}
That is because we can write $\C \H^n \setminus \C \H^{n-2} \cong \C \H^{n-2} \times \S^3 \times (0,\infty)$, and $\C \P^1 $ (the base of the Hopf fibration) has constant holomorphic curvature $4$.
So the purpose of this subsection is to prove the following:

\begin{lemma}
There do not exist functions $h, v, $ and $v_r$ that, when inserted into equation \eqref{eqn:/g}, yield a complete finite volume Riemannian metric on $\C \H^n \setminus \C \H^{n-2}$ with nonpositive sectional curvature.
\end{lemma}

\begin{proof}
Since we are removing a copy of $\C \H^{n-2}$ from $\C \H^n$, for the metric to be complete we need to define $h, v, $ and $v_r$ on $(- \infty, \infty)$.  
These functions need to be positive for the metric to be Riemannian, and they need to be non-decreasing for there to be any chance of nonpositive curvature.  
For any hope of finite volume, we need all of the following limits to hold:
	\begin{equation*}
	\lim_{r \to - \infty} h, h', v, v', v_r, v_r' = 0.
	\end{equation*}
	
Now, from the formula for the $R_{5656}$ term in Theorem \ref{thm:curvature tensor 2} we must have that
	\begin{align}
	\frac{4-(v')^2}{v^2} - \frac{3v_r^2}{4v^4} &\leq 0  \label{eqn:neg curv 1}  \\
	\iff	\hskip 30pt	\frac{16 - 4(v')^2}{3} &\leq \left( \frac{v_r}{v} \right)^2.  \nonumber
	\end{align}
In particular, we see that a necessary requirement for nonpositive curvature is that
	\begin{equation}\label{eqn:contradiction 1}
	\lim_{r \to - \infty} \frac{v_r}{v} > 1.
	\end{equation}
	
From the formula for the $R_{k9k9}$ term in Theorem \ref{thm:curvature tensor 2} we must have that
	\begin{align}
	- \frac{v' v_r'}{v v_r} + \frac{v_r^2}{4v^4} &\leq 0  \nonumber  \\
	\Longrightarrow	\hskip 30pt	\frac{3v_r^2}{4v^4} &\leq \frac{3 v' v_r'}{v v_r}.  \label{eqn:neg curv 2}
	\end{align}
Comparing equations \eqref{eqn:neg curv 1} and \eqref{eqn:neg curv 2}, we see that
	\begin{equation*}
	\frac{4 - (v')^2}{v^2} \leq \frac{3 v' v_r'}{v v_r}	\hskip 20pt	\Longrightarrow		\hskip 20pt	4 - (v')^2 \leq (3v'v_r') \cdot \frac{v}{v_r}
	\end{equation*}
is also a necessary requirement for nonpositive curvature.  
But as $r \to - \infty$, we know that $4 - (v')^2 \to 4$ and $3 v' v_r' \to 0$.  
Thus, we must have that
	\begin{equation}\label{eqn:contradiction 2}
	\lim_{r \to - \infty} \frac{v}{v_r} = \infty	\hskip 20pt	\Longrightarrow		\hskip 20pt	\lim_{r \to - \infty} \frac{v_r}{v} = 0.
	\end{equation}
Equations \eqref{eqn:contradiction 1} and \eqref{eqn:contradiction 2} provide a contradiction, proving the Lemma.
	
\end{proof}

\vskip 20pt 

%beginning of Section 5
\section{Preliminaries}\label{section:preliminaries}

\subsection{Formula for the curvature tensor of $\C \H^n$ in terms of the complex structure $J$}
The components of the (4,0) curvature tensor of the complex hyperbolic metric $g$ can be expressed in terms of $g$ and the complex structure $J$.  
The following formula can be found in \cite{KN} or in Section 5 of \cite{Belegradek complex} (recall Remark \ref{rmk:curvature notation} from the Introduction).  
In this formula $X, Y, Z$, and $W$ are arbitrary vector fields.
	\begin{align}
	 \la R_g&(X,Y)Z,W \ra_g = \la X,W \ra_g \la Y,Z \ra_g - \la X,Z \ra_g \la Y,W \ra_g  \label{eqn:J curvature formula}\\
	&+  \la X,JW \ra_g \la Y,JZ \ra_g - \la X,JZ \ra_g \la Y,JW \ra_g + 2 \la X,JY \ra_g \la W,JZ \ra_g.  \nonumber
	\end{align}

\subsection{Koszul's formula for the Levi-Civita connection}
Let $X, Y,$ and $Z$ denote vector fields on a Riemannian manifold $(M,g)$.
The following is the well-known ``Koszul formula" for the values of the Levi-Civita connection $\nabla$ (which can be found on pg.\ 55 of \cite{do Carmo})
	\begin{align}
	\la \nabla_Y X, Z \ra_g = & \, \frac{1}{2} ( X \la Y,Z \ra_g + Y \la Z,X \ra_g - Z \la X,Y \ra_g  \label{eqn:Koszul formula}  \\
	&- \la [X,Z],Y \ra_g - \la [Y,Z], X \ra_g - \la [X,Y],Z \ra_g ) . \nonumber
	\end{align}
In this paper we will usually be considering an orthonormal frame $(Y_i )$.
In this setting we know that $\la Y_i , Y_j \ra_g = \d_{ij}$, where $\d_{ij}$ denotes Kronecker's delta.
Therefore the first three terms on the right hand side of formula \eqref{eqn:Koszul formula} are all zero.  
Thus, in an orthonormal frame, formula \eqref{eqn:Koszul formula} reduces to
	\begin{equation}
	\la \nabla_Y X, Z \ra_g = - \frac{1}{2} \left( \la [X,Z],Y \ra_g + \la [Y,Z], X \ra_g + \la [X,Y],Z \ra_g \right).  \label{eqn:connection}
	\end{equation}

\subsection{General curvature formulas for warped product metrics}\label{subsection:product formulas}
The curvature formulas below, which were worked out by Belegradek in \cite{Belegradek real} and stated in Appendix B of \cite{Belegradek complex}, apply to metrics of the form $g = g_r + dr^2$ on manifolds of the form $E \times I $ where $I$ is an open interval and $E$ is a manifold.  
The formulas are true provided that for each point $q \in E$ there exists a local frame $\{ X_i \}$ on a neighborhood $U_q$ in $E$ which is $g_r$-orthogonal for each $r$.  
Such a family of metrics $(E, g_r)$ is called {\it simultaneously diagonalizable}.  
Let
	\begin{equation*}
	h_i(r) := \sqrt{g_r(X_i,X_i)}.
	\end{equation*}
Then the local frame $\{ Y_i \}$ defined by
	\begin{equation*}
	Y_i = \frac{1}{h_i} X_i
	\end{equation*}
is a $g_r$-orthonormal frame on $U_q$ for any value of $r$.  
We then have the following formulas for the (4,0) curvature tensor $R_g$ in terms of the (4,0) curvature tensor $R_{g_r}$, the collection $\{ h_i \}$, and the Lie brackets $[Y_i,Y_j]$.  
Note that $\la , \ra$ is used to denote the metric $g$ and $\dr = \ddr$.
	\begin{align}
	\la R_g(Y_i,Y_j)Y_i, Y_j \ra &= \la R_{g_r} (Y_i,Y_j)Y_i,Y_j \ra - \frac{h_i' h_j'}{h_i h_j}  \label{eqn:Bel 1}  \\
	\la R_g(Y_i,Y_j)Y_k, Y_l \ra &= \la R_{g_r} (Y_i, Y_j)Y_k, Y_l \ra  \hskip 10pt \text{if } \{i, j \} \neq \{ k, l \}  \label{eqn:Bel 2}  \\
	\la R_g (Y_i, \dr) Y_i, \dr \ra &= - \frac{h_i''}{h_i}  \hskip 40pt \la R_g(Y_i, \dr)Y_j, \dr \ra = 0 \hskip 10pt \text{if } i \neq j  \label{eqn:Bel 3}  \\
	2\la R \left( \dr, Y_i \right) Y_j, Y_k \ra &=  \la [Y_i,Y_k],Y_j \ra \left( \ln \frac{h_j}{h_k} \right)' + \la [Y_j,Y_i],Y_k \ra \left( \ln \frac{h_k}{h_j} \right)'  \label{eqn:Bel 4}  \\
	&+ \la [Y_j,Y_k],Y_i \ra \left( \ln \frac{h_i^2}{h_j h_k} \right)'   \nonumber .
	\end{align}

\subsection{The Nijenhuis Tensor}
In Sections \ref{section:chn/hn} and \ref{section:chn/chk} we will be explicitly dealing with $\C \H^n$.  
Since the almost complex structure on $\C \H^n$ is integrable, we have that the {\em Nijenhuis Tensor} is identically equal to zero.  
Explicitly, for any vector fields $X$ and $Y$ on $\C \H^n$, we have that
	\begin{equation}\label{eqn:Nijenhuis tensor}
	0 = [X,Y] + J[JX,Y] + J[X,JY] - [JX,JY]
	\end{equation}
where $J$ denotes the complex structure on $\C \H^n$.

\vskip 20pt

%beginning of Section 6
\section{Computations for the Lie brackets for $\C \H^3 \setminus \H^3$}\label{section:proof of Lie brackets 1}

\vskip 10pt

The whole purpose of this Section is to prove Theorem \ref{thm:Lie brackets 1}.

\begin{proof}[Proof of Theorem \ref{thm:Lie brackets 1}]
There are $5 \times 10 = 50$ structure constants to compute from equation \eqref{eqn:structure constants 1}.  
From equation \eqref{eqn:first Lie bracket 1} we know that $c_{45}^4 = 0$ and $c_{45}^5 = - \cot(\th)$, leaving 48 unknown structure constants.

We can combine formula \eqref{eqn:Bel 4} with equations \eqref{eqn:c3 tensor 1} through \eqref{eqn:c3 tensor 4} to compute many of the constants.  
As a first example, note that
	\begin{align*}
	0 &= 2R^{\bf c_3}_{6131} = 0 + 2 \la [Y_3,Y_1],Y_1 \ra_{\bf c_3} \left( \ln \frac{h}{h_r} \right)'  \\
	&= - \frac{2 c_{13}^1}{\cosh(2r)} \left( \ln \frac{\cosh(r)}{\cosh(2r)} \right)'  \\
	\end{align*}
and thus $c_{13}^1 = 0$.  
We can analogously show
	\begin{align*}
	0 &= c_{14}^1 = c_{15}^1 = c_{23}^2 = c_{24}^2 = c_{25}^2 = c_{13}^3 = c_{23}^3  \\
	&= c_{34}^3 = c_{35}^3 = c_{14}^4 = c_{24}^4 = c_{34}^4 = c_{15}^5 = c_{25}^5 = c_{35}^5.
	\end{align*}
This narrows us down to $32$ unknown constants.  

Continuing with the same formula and equations, we have that
	\begin{align*}
	0 &= 2R^{\bf c_3}_{6145} = 0  + 0 + \la [Y_4,Y_5],Y_1 \ra_{\bf c_3} \left( \ln \frac{\cosh^2(r)}{\sinh^2(r)} \right)'  \\
	&= \frac{c_{45}^1 h}{v^2} \left( \ln \frac{\cosh^2(r)}{\sinh^2(r)} \right)'
	\end{align*}
and therefore $c_{45}^1 = 0$.  
Analogously, $c_{45}^2 = c_{45}^3 = c_{12}^3 = c_{12}^4 = c_{12}^5 = 0$.  
This reduces us to $26$ unknowns.  
But we can also use the same curvature formulas here, but with the indices permuted, to derive some simple equations relating some of the constants.  
For example,
	\begin{align*}
	0 &= R^{\bf c_3}_{6415} = 0 + \la [Y_1,Y_4],Y_5 \ra_{\bf c_3} \left( \ln \frac{\sinh(r)}{\cosh(r)} \right)' + \la [Y_1,Y_5],Y_4 \ra_{\bf c_3} \left( \ln \frac{\sinh(r)}{\cosh(r)} \right)'  \\
	&= \frac{1}{h} (c_{14}^5 + c_{15}^4) \left( \ln \frac{\sinh(r)}{\cosh(r)} \right)'
	\end{align*}
and thus $c_{14}^5 = - c_{15}^4$.  
Analogously, we have the identities
	\begin{align*}
	c_{24}^5 = -c_{25}^4	&	\hskip 40pt	c_{34}^5 = - c_{35}^4	\hskip 40pt	c_{13}^2 = -c_{23}^1  \\
	&c_{14}^2 = - c_{24}^1	\hskip 40pt	c_{15}^2 = - c_{25}^1.
	\end{align*}
	
Combining formula \eqref{eqn:Bel 4} with the fact that $2R^{\bf c_3}_{6413} = 2$ gives that
	\begin{align*}
	2 &= \la [Y_4,Y_3],Y_1 \ra_{\bf c_3} \left( \ln \frac{h}{h_r} \right)' + \la [Y_1,Y_4],Y_3 \ra_{\bf c_3} \left( \ln \frac{h_r}{h} \right)' + \la [Y_1,Y_3],Y_4 \ra_{\bf c_3} \left( \ln \frac{v^2}{hh_r} \right)'  \\
	&= - \frac{c_{34}^1 h}{h_r v} \left( \ln \frac{h}{h_r} \right)' + \frac{c_{14}^3 h_r}{hv} \left( \ln \frac{h_r}{h} \right)' + \frac{c_{13}^4 v}{hh_r} \left( \ln \frac{v^2}{hh_r} \right)'  \\
	&= \left( \frac{-c_{34}^1 \cosh(r)}{\cosh(2r) \sinh(r)} - \frac{c_{14}^3 \cosh(2r)}{\cosh(r) \sinh(r)} \right) \left( \tanh(r) - 2 \tanh(2r) \right)  \\
	&+ \frac{c_{13}^4 \sinh(r)}{\cosh(r) \cosh(2r)} \left( 2 \coth(r) - \tanh(r) - 2 \tanh(2r) \right).
	\end{align*}
Now one consults equation (5.9) in \cite{Min} to see that the solutions to this equation are
	\begin{equation*}
	c_{13}^4 = 1	\hskip 40pt	c_{14}^3 = 1	\hskip 40pt	c_{34}^1 = -1.
	\end{equation*}
In exactly the same manner we can use $R^{\bf c_3}_{6253}$ with equation (5.9) in \cite{Min} to compute
	\begin{equation*}
	c_{23}^5 = 1	\hskip 40pt	c_{25}^3 = 1	\hskip 40pt	c_{35}^2 = -1.
	\end{equation*}
This leaves $20$ unknowns together with the $6$ identities listed above.  
Now, using $R^{\bf c_3}_{6135}$ we have that
	\begin{align*}
	0 &= \la [Y_1,Y_5],Y_3 \ra_{\bf c_3} \left( \ln \frac{h_r}{v} \right)' + \la [Y_3,Y_1],Y_5 \ra_{\bf c_3} \left( \ln \frac{v}{h_r} \right)' + \la [Y_3,Y_5],Y_1 \ra_{\bf c_3} \left( \ln \frac{h^2}{h_r v} \right)'  \\
	&= \left( \frac{c_{15}^3 \cosh(2r)}{\cosh(r) \sinh(r)} + \frac{c_{13}^5 \sinh(r)}{\cosh(r) \cosh(2r)} \right) \left( 2 \tanh(2r) - \coth(r) \right)  \\
	&+ \frac{c_{35}^1 \cosh(r)}{\cosh(2r) \sinh(r)} \left( 2 \tanh(r) - 2\tanh(2r) - \coth(r) \right).
	\end{align*}
One can check that the only solution to this equation (which holds for all values of $r$) is $c_{13}^5 = c_{15}^3 = c_{35}^1 = 0$.  
Analogously, we can use $\R^{\bf c_3}_{6234}$ to show that $c_{23}^4 = c_{24}^3 = c_{34}^2 = 0$.  
These equations reduce us to $14$ unknowns.  

This is as much information as we can gain from formula \eqref{eqn:Bel 4}.  
So we next turn to the Nijenhuis Tensor \eqref{eqn:Nijenhuis tensor}.  
First applying this to $(Y_1,Y_2)$, we have
	\begin{align*}
	0 &= [Y_1,Y_2] - J[Y_4,Y_2] - J[Y_1,Y_5] - [Y_4,Y_5]  \\
	&= \frac{1}{h} (c_{12}^1 Y_1 + c_{12}^2 Y_2) + J \left( \frac{c_{24}^1}{v} Y_1 + \frac{c_{24}^5}{h} Y_5 \right) - J \left( \frac{c_{15}^2}{v} Y_2 + \frac{c_{15}^4}{h} Y_4 \right) + \frac{1}{v} \cot(\th) Y_5  \\
	&= \frac{1}{h} \left( c_{12}^1 - c_{15}^4 \right) Y_1 + \frac{1}{h} (c_{12}^2 + c_{24}^5) Y_2 - \frac{c_{24}^1}{v} Y_4 + \frac{1}{v} (c_{15}^2 + \cot(\th)) Y_5.
	\end{align*}
Therefore, we have that 
	\begin{equation*}
	c_{24}^1 = 0 = -c_{14}^2 \hskip 25pt c_{15}^2 = -\cot(\th) = -c_{25}^1  \hskip 25pt  c_{12}^1 = c_{15}^4 = -c_{14}^5  \hskip 25pt  c_{12}^2 = -c_{24}^5.  
	\end{equation*}
We can also apply the Nijenhuis tensor to the pairs $(Y_1,Y_3)$ and $(Y_2,Y_3)$, but these are much less productive.  
These applications only give us the pair of identities
	\begin{equation*}
	c_{13}^2 = - c_{34}^5	\hskip 50pt	c_{23}^1 = - c_{35}^4
	\end{equation*}
the former of which comes from the pair $(Y_1,Y_3)$, and the latter from the pair $(Y_2,Y_3)$.  

At this stage, we have reduced our $10$ Lie brackets as follows:
	\begin{align*}
	&[X_1,X_2] =  c_{12}^1 X_1 + c_{12}^2 X_2  		\hskip 64pt  	[X_1,X_3] =  c_{13}^2 X_2 + X_4  \\
	&[X_1,X_4] =  X_3 - c_{12}^1 X_5  	\hskip 77pt	[X_1,X_5] =  - \cot(\th) X_2 + c_{12}^1 X_4  \\
	&[X_2,X_3] =  -c_{13}^2 X_1 + X_5  	\hskip 70pt	[X_2,X_4] =  -c_{12}^2 X_5  \\
	&[X_2,X_5] =  \cot(\th) X_1 + X_3 + c_{12}^2 X_4  	\hskip 26pt	[X_3,X_4] =  -X_1 - c_{13}^2 X_5  \\
	&[X_3,X_5] =  -X_2 + c_{13}^2 X_4  	\hskip 70pt	[X_4,X_5] =  - \cot(\th) X_5.
	\end{align*}
Notice that, using the known identities, we can reduce the system to three unknowns:  $c_{12}^1, c_{12}^2$, and $c_{13}^2$.  
All that is left is to show that $c_{12}^1 = \pm 1$, $c_{12}^2 = 0$, and $c_{13}^2 = \mp \cot(\th)$.  

At this point we have exhausted all of our ``easy" options.  
The only way to obtain new relationships between the structure constants is to compute new components of $R_{\bf c_3}$.  
To do this, one needs to first use equation \eqref{eqn:Koszul formula} with the values for the Lie brackets given above to compute the Levi-Civita connection $\nabla$ compatible with ${\bf c_3}$.  
Of course, these formulas will contain the constants $c_{12}^1, c_{12}^2$, and $c_{13}^2$.  
But when the correct values for these constants are inserted, these formulas will reduce to those of Theorem \ref{thm:connection 1}.
Then once one has computed $\nabla$, they can use those values to compute the components of $R_{\bf c_3}$.  

The first component that will be useful is $R^{\bf c_3}_{1212}$:
	\begin{align*}
	-1 = &R^{\bf c_3}_{1212} = \la \nabla_{Y_2} \nabla_{Y_1} Y_1 - \nabla_{Y_1} \nabla_{Y_2} Y_1 + \nabla_{[Y_1,Y_2]} Y_1,Y_2 \ra_{\bf c_3}  \\
	= &\la \nabla_{Y_2} \left( \frac{- c_{12}^1}{\cosh(r)} Y_2 - \tanh(r) Y_6 \right) - \nabla_{Y_1} \left( \frac{-c_{12}^2}{\cosh(r)} Y_2 \right) \\
	&+ \frac{c_{12}^1}{\cosh(r)} \nabla_{Y_1}Y_1 + \frac{c_{12}^2}{\cosh(r)} \nabla_{Y_2}Y_1, Y_1 \ra_{\bf c_3}  \\
	= &- \frac{\sinh^2(r)}{\cosh^2(r)} - \frac{((c_{12}^1)^2 + (c_{12}^2)^2)}{\cosh^2(r)}  \\
	&\Longrightarrow \hskip 30pt  (c_{12}^1)^2 + (c_{12}^2)^2 = 1.
	\end{align*}
	
The next component that we use is $R^{\bf c_3}_{1512}$.  
We will skip the details and just note that
	\begin{equation*}
	0 = R^{\bf c_3}_{1512} = \frac{c_{12}^2}{\sinh(r) \cosh(r)} \cdot \cot(\th)
	\end{equation*}
which implies that $c_{12}^2 = 0$.  
Combining this with the first equation shows that $c_{12}^1 = \pm 1$.  
Finally, to compute $c_{13}^2$ we use $R^{\bf c_3}_{1412}$:
	\begin{equation*}
	0 = R^{\bf c_3}_{1412} = \frac{- c_{13}^2}{\sinh(r) \cosh(r)} - \frac{c_{12}^1}{\sinh(r) \cosh(r)} \cdot \cot(\th).
	\end{equation*}
Therefore
	\begin{equation*}
	c_{13}^2 = - (\pm 1) \cot(\th) = \mp \cot(\th).
	\end{equation*}

\end{proof}

\subsection*{Acknowledgements}  
The author would like to thank J.F. Lafont for both suggesting this problem and for many helpful discussions about the work contained in this paper.
The author would also like to thank I. Belegradek, J. Meyer, and B. Tshishiku for various comments which aided in this research.

\end{document}